\documentclass[11pt]{article}
\usepackage[a4paper,top=2cm,bottom=2cm,left=1.5cm,right=1.5cm]{geometry}

\usepackage{microtype}
\usepackage{float}
\usepackage{algorithm}
\usepackage{algorithmicx}
\usepackage{algpseudocode}

\usepackage{graphicx}
\usepackage{caption}
\usepackage{subcaption}
\usepackage{booktabs} 
\usepackage{authblk}

\usepackage{hyperref}
\usepackage{natbib}

\usepackage{amsmath}
\usepackage{amssymb}
\usepackage{mathtools}
\usepackage{amsthm}
\usepackage{xargs}
\usepackage{dsfont, mathtools, nicefrac}
\theoremstyle{plain}
\newtheorem{theorem}{Theorem}[section]
\newtheorem{proposition}[theorem]{Proposition}
\newtheorem{lemma}[theorem]{Lemma}
\newtheorem{corollary}[theorem]{Corollary}
\theoremstyle{definition}
\newtheorem{definition}[theorem]{Definition}

\theoremstyle{remark}

\usepackage[textsize=tiny]{todonotes}

\newcommand{\rmd}{\mathrm{d}}
\newcommand{\rme}{\mathrm{e}}
\newcommand{\eqsp}{\;}
\newcommand{\tv}{\mathsf{d}_{\mathrm{tv}}}

\renewcommand{\mid}{\;\ifnum\currentgrouptype=16 \middle\fi|\;}
\newcommand*{\bX}{\mathbf{X}}
\newcommand*{\bx}{\mathbf{x}}
\newcommand*{\by}{\mathbf{y}}
\newcommand*{\Xset}{\mathbb{X}}
\newcommand*{\Xtribe}{\mathcal{X}}
\newcommand{\trunc}{T}
\newcommand{\invtrunc}{\mathsf{k}}

\newcommandx*\Hmean[3][1=\theta_{n}, 2=\tau(\invtrunc_{n+1}), 3=\bX^{(n+1)}]{
H_{#1}^{(#2)}\mathopen{}\mathclose\bgroup\left(#3\aftergroup\egroup\right)%
}
\newcommandx*\Hhat[3][1=\theta_{n}, 2=K_{n+1}, 3=\bX^{(n+1)}]{%
\widehat{H}_{#1}^{(#2)}\mathopen{}\mathclose\bgroup\left(#3\aftergroup\egroup\right)%
}

\newcommandx*\Hmlmc[2][1=n, 2=n+1]{%
\widehat{H}_{\theta_{#1}}^{(K_{#2})}\mathopen{}\mathclose\bgroup\left(\bX^{(#2)}\aftergroup\egroup\right)%
}

\newcounter{hypH}
\newenvironment{hypH}{
    \refstepcounter{hypH}
    \begin{itemize}
    \item[{\bf H\arabic{hypH}}]
    }
{\end{itemize}}

\newcounter{hypA}
\newenvironment{hypA}{
    \refstepcounter{hypA}
    \begin{itemize}
    \item[{\bf A\arabic{hypA}}]
    }
{\end{itemize}}

\author[1]{Antoine Godichon-Baggioni}
\author[2]{Gabriel Lang}
\author[1]{Sylvain Le Corff}
\author[3]{Julien Stoehr}
\author[1,4]{Sobihan Surendran}

\affil[1]{Laboratoire de Probabilités, Statistique et Modélisation, Sorbonne Université}
\affil[2]{UMR MIA-Paris, Université Paris-Saclay, AgroParisTech}
\affil[3]{Université Paris-Dauphine, Université PSL, CNRS, CEREMADE}
\affil[4]{LOPF, Califrais’ Machine Learning Lab}

\begin{document}

\title{Convergence of Multi-Level Markov Chain Monte Carlo Adaptive Stochastic Gradient Algorithms}

\date{}
\maketitle




\begin{abstract}
Stochastic optimization in learning and inference often relies on Markov chain Monte Carlo (MCMC) to approximate gradients when exact computation is intractable. However, finite-time MCMC estimators are biased, and reducing this bias typically comes at a higher computational cost.
We propose a multilevel Monte Carlo gradient estimator whose bias decays as $\mathcal{O}(T_n^{-1})$ while its expected computational cost grows only as $\mathcal{O}(\log T_n)$, where $T_n$ is the maximal truncation level at iteration $n$.
Building on this approach, we introduce a multilevel MCMC framework for adaptive stochastic gradient methods, leading to new multilevel variants of Adagrad and AMSGrad algorithms.
Under conditions controlling the estimator bias and its second and third moments, we establish a convergence rate of order $\mathcal{O}(n^{-1/2})$ up to logarithmic factors.
Finally, we illustrate these results on Importance-Weighted Autoencoders trained with the proposed multilevel adaptive methods.
\end{abstract}


\section{Introduction}
A central problem in optimization is to approximate the minimum of a 
function $V : \mathbb{R}^d \to \mathbb{R}$. 
When the function is differentiable, 
gradient-based methods are the standard approach \citep{Nesterov2004,BoydVandenberghe2004}.
In stochastic optimization, however, the exact gradient of $V$ is typically unavailable, and only stochastic gradient estimators are accessible, which are often unbiased.
In this setting, stochastic gradient methods
\citep{robbins1951stochastic, BottouCurtisNocedal2018}
are commonly employed and can be viewed as instances of the
Robbins-Monro stochastic approximation scheme. 
These methods are particularly popular since they scale well to large datasets and naturally accommodate sequential or streaming data.
Their theoretical properties have been extensively studied in recent years, both in the convex setting \citep[\textit{e.g.}][]{Nemirovski2009} , and in the non-convex settings \citep[\textit{e.g.}][]{GhadimiLan2013, mai2020convergence}.

Standard stochastic gradient methods have two classical limitations: (i) they typically require unbiased gradient estimators, and (ii) they use a single scalar step size applied uniformly across coordinates, which can be inadequate for ill-conditioned optimization problems. 
Recent work has investigated stochastic gradient schemes
with biased estimators in non-adaptive settings \citep[\textit{e.g.}][for projected stochastic gradient scheme]{karimi2019non, ajalloeian2020convergence, hu2021bias, dieuleveut2023stochastic, demidovich2023guide, batardiere2025}.
To overcome the limitations of a scalar step size when unbiased gradient estimators are available, adaptive algorithms replace it with a matrix-valued preconditioner.
When this matrix is designed to approximate the square root of the inverse gradient variance, the resulting method is referred to as Full Adagrad \citep{duchi2011adaptive,godichon2024full}. 
Restricting the preconditioning matrix to its diagonal yields the widely used  Adagrad algorithm 
\citep{duchi2011adaptive}. 
Finally, when the preconditioning matrix estimates the inverse of 
the Hessian, one obtains a stochastic Newton type algorithm 
\citep{BottouCurtisNocedal2018,Gower2019,BGBP2019,boyer2020asymptotic}.

In recent work, these two limitations have been addressed jointly through the analysis of convergence rates for adaptive stochastic gradient methods with biased gradient estimators \citep{surendran2024non}.
Such scenarios arise for instance when gradients are approximated via Markov chain Monte Carlo procedures.
In the non-convex setting, existing results yield convergence rates of $\mathcal{O}(n^{-1/2}+T_n^{-1})$, where $n$ is the number of iterations and $T_n$ is the Markov chain length used at iteration $n$. 
Consequently, recovering the standard convergence rate $\mathcal{O}(n^{-1/2})$ requires choosing $T_n = n^{1/2}$, leading to an overall computational complexity of $\mathcal{O}(n^{3/2} d)$, which can be prohibitively large in practice. 

In this work, we aim to estimate gradients using Multi-Level Monte Carlo \citep[MLMC,][]{giles2008multilevel, dorfman2022adapting} within an adaptive stochastic optimization framework.
Indeed, MLMC methods yield Markov-chain-based gradient estimators with bias of order $T^{-1}$, while the expected per-iteration cost scales as $\mathcal{O}(d\log T)$.
This improvement comes at the expense of increased variance in the gradient estimates.
Understanding how adaptive preconditioning interacts with both bias and variance is therefore essential, from both theoretical and practical perspectives.
We show that the usual convergence rate $\mathcal{O}(n^{-1/2})$ can be achieved (up to logarithmic factors) while keeping the overall computational budget very low, of order $\mathcal{O}(dn\log n)$.

The main contributions of this paper can be summarized as follows.
\begin{itemize}
    \item  Following \citet{dorfman2022adapting}, we introduce a novel multilevel Markov chain Monte Carlo framework tailored to adaptive stochastic gradient methods.
    \item This framework  enables the development of new multilevel variants of state-of-the-art adaptive algorithms, including Adagrad and AMSGrad, which is an exponential moving average variant of Adam, when gradients are estimated using Markov chain–based procedures.
    \item We establish explicit convergence under conditions controlling the bias and the second and third moments of the gradient estimator.
    \item We illustrate these theoretical results with numerical experiments on Importance-Weighted Autoencoders trained on the CIFAR-10 dataset, using several Markov chain–based gradient estimation schemes.
\end{itemize}
The paper is organized as follows. 
Section~\ref{sec:background} provides a background on multi-level Monte Carlo and adaptive optimization.  
Section~\ref{sec:convMLMC} establishes general convergence results for adaptive gradient methods coupled with MLMC techniques, along with convergence guarantees of MLMC variants of Adagrad and AMSGrad algorithms. 
Section~\ref{sec:exp} illustrates the theoretical results in the context of Importance Weighted Autoencoders.

\section{Multi-level adaptive optimization}
\label{sec:background}

\paragraph{Notations. } 
For any matrix $M$,  $\text{Diag}(M)$ denotes the diagonal matrix formed from the diagonal entries of $M$, and let $\lambda_{\min}(M)$ and $\lambda_{\max}(M)$ denote its smallest and largest eigenvalues, respectively. Given diagonal matrices $A,B$, $\min\{A,B\}$ and $\max\{A,B\}$ denote the diagonal matrix $D$ with respectively $D_{ii} = \min \lbrace A_{ii}, B_{ii} \rbrace$ and $D_{ii} = \max \lbrace A_{ii}, B_{ii} \rbrace$.
The vector $e_i$ denotes the $i$-th canonical basis vector. For $a, b \in \mathbb{R}$, we write $a \wedge b = \min\{a,b\}$ and $a \vee b = \max\{a,b\}$.

Let $(\Omega,\mathcal{F},\mathbb{P})$ be a probability space and $(\Xset, \Xtribe) \subseteq (\mathbb{R}^q, \mathcal{B}(\mathbb{R}^q))$ be a measurable space.
The total variation distance between two distributions $\mu$ and $\nu$ on $(\Omega,\mathcal{F})$ is defined as
\begin{equation*}
\tv(\mu,\nu) = \sup_{A\in\mathcal{F}} \left|\mu(A) - \nu(A)\right|\eqsp.
\end{equation*}
For any distribution $\mu$, we write $\mathbb{E}_\mu$ the expectation under $\mu$. 
Consider a Markov chain $\bX = (X_k)_{k\geq 1}$ on $\Xset$, with transition Markov kernel $P$ and stationary distribution $\pi$, starting from $X_1 = x$.
For any measurable functions $f$ and $k\geq 1$, we denote the expectation with respect to the $k$-step transition kernel, 
\begin{equation*}
\mathbb{E}\left[f(X_{k+1})\right] = \int P^{k}(x, \rmd y)f(y)\eqsp,
\end{equation*}
and the expectation with respect to the invariant distribution
\begin{equation*}
\mathbb{E}_{\pi}\left[f(X_{k+1})\right] = \int f(x)\pi(\rmd x)\eqsp.
\end{equation*}

\paragraph{Adaptive stochastic approximation. } 
Consider a general optimization problem:
\begin{equation}      
\label{pb_intro}
\theta^*\in\underset{\theta \in \mathbb{R}^d}{\mathrm{argmin}} \, V(\theta)\eqsp,
\end{equation}
where $V$ is the objective function.
Stochastic Approximation (SA) methods are sequential algorithms designed to find the zeros of a function when only noisy observations are available. In \citet{robbins1951stochastic}, the authors introduced the Stochastic Approximation algorithm as a recursive algorithm to solve the following integration equation with respect to a probability measure $\pi_{\theta}$ on $(\Xset,\Xtribe)$:
\begin{equation}   \label{mean_field}
h(\theta) =  \mathbb{E}_{\pi_\theta}\left[H_{\theta}(X)\right] = 0\eqsp,
\end{equation}
where $h$ refers to the mean field function. 
If $H_{\theta}(X)$ is an unbiased estimator of the gradient of the objective function, then $h(\theta) = \nabla V(\theta)$. As a result, the minimization problem \eqref{pb_intro} is then equivalent to solving problem \eqref{mean_field}, Stochastic Gradient Descent being a specific instance of SA. 
Given $\theta_0\in\mathbb{R}^d$ and a positive sequence $(\gamma_n)_{n\geq  1}$, SA methods then define a sequence $(\theta_n)_{n\geq 0}$ as follows:
\begin{equation*}
\theta_{n+1}=\theta_{n}-\gamma_{n+1} H_{\theta_{n}}\left(\bX^{(n+1)}\right)\eqsp,
\end{equation*}
where $\bX^{(n+1)}$ is a set of random variables taking values in $(\Xset,\mathcal{X})$, such as a sample from $\pi_{\theta}$ for instance.
The term $H_{\theta_{n}}\left(\bX^{(n+1)}\right)$ is the $n$-th stochastic update, also referred to as the drift term, and is a potentially biased estimator of $\nabla V(\theta_n)$.

In this paper, we rather consider stochastic adaptive algorithm of the form
\begin{equation}
\label{eq:asgd-updt}
\theta_{n+1} = \theta_{n} - \gamma_{n+1}A_{n} H_{\theta_{n}} \left( \bX^{(n+1)} \right)\eqsp,
\end{equation}
where $A_n$ is an adaptive preconditioning matrix.
Several standard choices fit this framework. A common example is full Adagrad \cite{duchi2011adaptive}, which is based on
\begin{equation*}
\overline{\mathcal H}_n(\bX^{(1:n+1)},\theta_{0:n}) = \frac{1}{n+1} \sum_{k=0}^{n} H^{\otimes 2}_{\theta_{k}}(\bX^{(k+1)})\eqsp,
\end{equation*}
where $H^{\otimes 2}_{\theta_{k}}(\bX^{(k+1)}) = H_{\theta_{k}}(\bX^{(k+1)}) H_{\theta_{k}}(\bX^{(k+1)})^\top$.
The adaptive matrix is then defined as
\begin{equation*}
A_n = \left\{\delta \mathit{I}_d +  \overline{\mathcal H}_n(\bX^{(1:n+1)},\theta_{0:n})
\right\}^{-1/2}\eqsp.
\end{equation*}
The latter is computationally demanding, and computing the square root of the inverse can be very sensitive in high dimensional settings. 
Alternatively, Adagrad with diagonal matrices is defined as:
\begin{equation}
\label{eq:def:An}
A_n = \left\{\delta \mathit{I}_d + \text{Diag}\left( \overline{\mathcal H}_n(\bX^{(1:n+1)},\theta_{0:n}) \right) \right\}^{-1/2}.
\end{equation}
In RMSProp \cite{tieleman2012lecture}, second-moment term $\overline{\mathcal H}_n(\bX^{(1:n+1)},\theta_{0:n})$ in \eqref{eq:def:An} is replaced by an exponential moving average of the past squared gradients, given by
\begin{equation*}
(1-\rho) \sum_{k=0}^{n} \rho^{n-k} H^{\otimes 2}_{\theta_{k}}(\bX^{(k+1)})\eqsp,
\end{equation*} 
where $\rho$ is the moving average parameter. 
Furthermore, when $A_n$ is a recursive estimate of the inverse Hessian, it corresponds to the Stochastic Newton algorithm \citep{boyer2020asymptotic}.

\paragraph{Adaptive algorithms with Multi-level Monte Carlo. } 
In a wide range of applications, gradient estimates are obtained using Markov chain-based algorithms; see, for instance, \citep{jaakkola1993convergence,baxter2001infinite,bhandari2018finite,cardoso2022br}.
More precisely, at iteration $n \geq 1$ and given the current parameter estimate $\theta_{n}$, consider a Markov chain
$\bX^{(n+1)} = (X_i^{(n+1)})_{i \geq 1} \in \Xset^{\mathbb{N}}$
with transition kernel $P_{\theta_{n}}$.
If the Markov kernel admits $\pi_{\theta_n}$ as unique invariant probability, a natural estimator of $h(\theta_n)$, as defined in \eqref{mean_field}, is given for $p \in \mathbb{N}_0$ by
\begin{equation*}
\Hmean[\theta_n][p][\bX^{(n+1)}] = \frac{1}{p} \sum_{i = 1}^p H_{\theta_n}\left(X_i^{(n+1)}\right)\eqsp.
\end{equation*}
To ensure that the bias of this estimator vanishes over time, ideally at a rate of order $1/n$ (or $1/\sqrt{n}$), one must, at each iteration, generate a Markov chain of length $n$ (or $\sqrt{n}$), which can be prohibitively expensive from a computational point of view.
Consequently, standard approaches often have to settle for a constant bias \citep{surendran2024non, surendran2024theoretical}. 
To address this limitation, we instead focus on Multi-Level Monte Carlo (MLMC) methods. 
These methods can achieve a bias of order $1/n$ (or $1/\sqrt{n}$) while requiring, on average, only a chain of length $\mathcal{O}(\log n)$.
The trade-off, however, is that they typically exhibit a larger variance, which depends on the mixing time of the underlying Markov chain \citep{dorfman2022adapting}.

In what follows, let $(T_n)_{n\geq 1}$ be a positive sequence, and $(K_n)_{n\geq 1}$ be a sequence of independent random variables identically distributed according to the geometric distribution $\mathcal{G}(1/2)$ such that for all $n \geq 1$, $K_n$ and $\bX^{(n)}$ are independent.
For any $n \geq 1$, writing $t_{n} = 2^{K_{n}}$, we define the MLMC estimator
\begin{equation}
\label{eq:mlmc-estim}
\Hmlmc[n-1][n] = H_{\theta_{n}}\left(X_1^{(n)}\right) + 
t_{n} \Big\{ 
\Hmean[\theta_{n-1}][t_{n}][\bX^{(n)}] -
\Hmean[\theta_{n-1}][t_{n}/2][\bX^{(n)}]
\Big\}\mathds{1}_{t_{n} \leq T_{n}}\eqsp.
\end{equation}
The MLMC-based adaptive stochastic gradient algorithm for a generic matrix $A_n$ is then described in Algorithm~\ref{alg:MLMC}.


\begin{algorithm}[tb]
\caption{MLMC--Adaptive SGD}
\label{alg:MLMC}
\begin{algorithmic}
\State {\bfseries Input:} $\theta_0 \in \mathbb{R}^d$, $N \in \mathbb{N}_0$, positive sequence $\left( \gamma_n \right)_{n \ge 1}$, positive sequence $\left( T_{n} \right)_{n \geq 1}$
\For{$n = 0$ {\bfseries to}  $N-1$}
\State Sample $K_{n+1} \sim \mathcal{G}(1/2)$ 
\State Sample a Markov chain $\bX^{(n+1)}$ with transition kernel $P_{\theta_{n}}$
\State Compute $\Hmlmc$ as defined in \eqref{eq:mlmc-estim}
\State Set
\[
\theta_{n+1} = \theta_{n} - \gamma_{n+1} A_{n}\Hmlmc
\]
\EndFor
 \end{algorithmic}
\end{algorithm}

\section{Convergence results of MLMC-Adaptive stochastic gradient algorithms}
\label{sec:convMLMC}

Let $\mathcal{F}_0$ denote the trivial filtration, and, for $n \geq 1$, define $\mathcal{F}_{n} = \sigma(K_1, \bX^{(1)}, \ldots, K_{n}, \bX^{(n)})$ the filtration generated by the process up to time $n$.
Consider the following assumptions.
\begin{hypA}\label{A-smooth}
The objective function is $L$-smooth: for all $\theta , \theta ' \in \mathbb{R}^{d}$, 
\begin{equation*}
\left\lVert \nabla V (\theta) - \nabla V(\theta ') \right\rVert \leq L \left\lVert \theta - \theta ' \right\rVert\eqsp.
\end{equation*}
\end{hypA}
\begin{hypA}\label{A-bounded}
There exists a positive constant $G$ such that for all $\theta \in \mathbb{R}^{d}$, 
\begin{equation*}
\left\lVert H_{\theta}\right\lVert_{\infty} \leq G\eqsp.
\end{equation*}
\end{hypA}
\begin{hypA}
\label{A-bias-var}
There exist constants $c_1,c_2$, such that for all $n \geq 0$, 
\begin{align*}
&\left\lVert \mathbb{E} \left[
     \Hmlmc \mid \mathcal{F}_{n}  
    \right] -   h\left(\theta_{n}\right) \right\rVert  \leq \frac{c_1 G }{T_{n+1}}\eqsp,\\
     &\mathbb{E}\left[
    \left\lVert 
     \Hmlmc   \right\rVert^{2}
    \right]  \leq c_2G^{2}\log(T_{n+1})\eqsp.
\end{align*}
\end{hypA}
\begin{hypA}
\label{A-moment-3}
There exists a constant $c_3$, such that for all $n \geq 0$, 
\begin{equation*}
\mathbb{E}\left[
    \left\lVert 
     \Hmlmc   \right\rVert^{3}
    \right]  \leq   c_3G^{3}\sqrt{T_{n+1}}\eqsp.
\end{equation*}
\end{hypA}

Assumptions A\ref{A-smooth} and A\ref{A-bounded} are classical in adaptive stochastic optimization \citep{defossezsimple, ward2019adagrad, reddi2019convergence}.
Assumptions A\ref{A-bias-var} and A\ref{A-moment-3} characterize the accuracy of the MLMC gradient estimator under a maximal truncation level $T_n$ and are verified for common MCMC schemes under standard ergodicity conditions (see Section \ref{sec:check-A-moments}). In particular, A\ref{A-bias-var} captures the bias–variance trade-off induced by maximal truncation index $T_n$: increasing $T_n$ reduces the bias at rate $T_n^{-1}$ but may increase the magnitude of the second and higher moments (e.g., through logarithmic and polynomial dependences on $T_n$), reflecting heavier tails for longer runs (see Lemma \ref{lemma:mom-2p-gen}).
A central goal of the convergence analysis is therefore to identify a growth rate for $T_n$ that drives the bias to zero while preventing the second moment and, in the case of Adagrad, also the third-moment, terms from dominating the convergence bounds.

In smooth non-convex optimization, although SA methods such as Algorithm \ref{alg:MLMC} are used in practice, convergence results are commonly established via a randomized variant of SA \citep[see][ for instance]{GhadimiLan2013, karimi2019non}.
This randomization is introduced solely for theoretical purposes, as it yields bounds in expectation on the gradient norm of the objective function $V$ that are optimal in the absence of global optimality assumptions.
More formally, throughout the remainder of the paper, consider Algorithm \ref{alg:MLMC} (or its variant) run for $N$ iterations. Rather than deriving bounds at the final iterate $\theta_N$, we establish bounds for a randomly selected iterate $\theta_R$, where, for any $N \geq 1$, $R$ is a discrete random variable in $\lbrace 0 , \ldots   N \rbrace$  such that for all $n\in\{0 , \ldots ,N\}$, 
\begin{equation}
\label{eq:def-r}
\mathbb{P}\left[ R = n \right] = \frac{\gamma_{n+1}\lambda_{n+1}}{\varpi_N},
\quad
\varpi_N = \sum_{k=0}^{N} \gamma_{k+1}\lambda_{k+1}\eqsp,
\end{equation}
where $(\lambda_{n})_{n \geq 1}$ is a sequence that depends on the specific choice of $A_n$ in Algorithm \ref{alg:MLMC}.
In the following, we assume that the mean field coincides with the gradient of an objective function, i.e., $h(\theta) = \nabla V (\theta)$ for all $\theta \in \mathbb{R}^{d}$, so that the only source of bias considered here is the one induced by the Markov chain. Other sources of bias are well studied and are therefore not discussed \citep[see][for instance]{demidovich2023guide, surendran2024non}.

\subsection{Convergence in the general preconditioned case}

We first present a general convergence analysis under the following additional assumption on the preconditioning matrices.

\begin{hypA}\label{A-eigen}
For all $n\geq 0$, the conditioning matrices $A_{n}$ used in \eqref{eq:asgd-updt} are $\mathcal{F}_{n}$-measurable, positive, and there exist $\underline{\lambda}_{n+1}$ and $\overline{\lambda}_{n+1} $ such that,
\begin{equation*}
\underline{\lambda}_{n+1} \leq  \lambda_{\min} \left( A_{n} \right) \leq \lambda_{\max} \left( A_{n} \right) \leq \overline{\lambda}_{n+1}\eqsp.
\end{equation*}
\end{hypA}

Since the exact form of $A_n$ is left unspecified, some control on its spectrum is required. Assumption~A\ref{A-eigen} is therefore crucial: without such bounds, the preconditioning may become ill-conditioned and the resulting estimates can diverge. 
For many standard adaptive schemes, including Adagrad and AMSGrad, these spectral bounds can be verified under Assumption~A\ref{A-bounded}.
The following theorem establishes a convergence bound for the randomized iterate in the general preconditioned setting.

\begin{theorem}\label{theo::gen}
Assume that A\ref{A-smooth} -- A\ref{A-bias-var} and A\ref{A-eigen} hold. 
For any $N \geq 1$, let $R$ be as defined in \eqref{eq:def-r} with $(\lambda_n)_{n\geq 1} = (\underline{\lambda}_n)_{n \geq 1}$.
Then,
\[
\mathbb{E}\left[ 
\left\lVert \nabla V (\theta_{R})   \right\rVert^{2} 
\right] \leq   
\frac{2}{\varpi_N} \Big\lbrace
\mathbb{E}\left[ V \left( \theta_{0} \right) \right] -  V \left( \theta^{*} \right)  
  + c_{1}G^{2}\sum_{n=1}^{N+1} \gamma_{n}\frac{\overline{\lambda}_{n}}{T_{n}} 
  + c_2 G^2 \frac{L}{2} \sum_{n = 1}^{N+1} \gamma_{n}^2 \overline{\lambda}_{n}^2 \log(T_{n})
  \Big\rbrace\eqsp.
\]
\end{theorem}

\begin{corollary}
\label{cor:gen}
If assumptions of Theorem \ref{theo::gen} hold, and if for all $n \geq 1$,
\begin{equation*}
\gamma_n \propto n^{-\gamma},
\quad \underline{\lambda}_{n} \propto n^{\underline{\lambda}},
\quad \overline{\lambda}_{n} \propto n^{\overline{\lambda}},
\quad T_n \propto n^{\alpha},
\end{equation*}
with $\alpha > 0$, and $\gamma$, $\underline{\lambda}$, $\overline{\lambda}$ non-negative constants such that $\gamma + \underline{\lambda} < 1$, then
\[
\mathbb{E}\left[ 
\left\lVert \nabla V (\theta_{R})   \right\rVert^{2} 
\right] = \mathcal{O}\bigg( 
\frac{1}{N^{1-\gamma-\underline{\lambda}}} \Big\lbrace
1 + \Psi_N(\alpha + \gamma - \overline{\lambda})
  + \Phi_N(2\gamma - 2\overline{\lambda})
  \Big\rbrace
  \bigg)\eqsp,
\]
where, for all $\eta \in \mathbb{R}$,
\begin{align}
\label{def::psi} \Psi_N(\eta) & =
\begin{cases}
N^{1 - \eta} & \text{ if} \quad\eta < 1,
\\
\log N & \text{ if} \quad\eta = 1,
\\
1 & \text{ if} \quad\eta > 1,
\end{cases}
\\
\label{def::phi}\Phi_N(\eta) & =
\begin{cases}
N^{1 - \eta}\log N & \text{ if} \quad\eta < 1,
\\
(\log N)^2 & \text{ if} \quad\eta = 1,
\\
1 & \text{ if} \quad\eta > 1.
\end{cases}
\end{align}
\end{corollary}

It follows from Corollary~\ref{cor:gen} that, if spectra of adaptive matrices admit uniform positive lower and upper bounds, namely
$\underline{\lambda} = \overline{\lambda} = 0$,
then choosing $\gamma = 1/2$ and $\alpha \geq 1/2$ yields
\begin{equation*}
\mathbb{E} \left[ \left\Vert \nabla V (\theta_{R}) \right\rVert^{2} \right] = \mathcal{O} \left( \frac{(\log N)^{2}}{\sqrt{N}} \right) ,
\end{equation*}
which matches the standard convergence rate for smooth non-convex objectives, up to an additional logarithmic factor.

\subsection{Convergence bounds for Adagrad and AMSGrad}

We now consider Algorithm~\ref{alg:MLMC} under the setting where, at iteration $n$, the conditioning matrix $A_n$ is not $\mathcal{F}_n$-measurable.
This situation arises for standard adaptive choices such as Adagrad and AMSGrad, where $A_n$ is computed using the truncation variable $K_{n+1}$ and the Markov chain $\bX^{(n+1)}$ generated at iteration $n$.

\paragraph*{Adagrad. } 
In the Adagrad version of Algorithm~\ref{alg:MLMC}, referred to as MLMC-Adagrad, the conditioning matrix $A_n$ is defined at iteration $n$ by
\begin{equation}
\label{eq:a-adagrad}
A_{n} = \left\lbrace \frac{1}{\varepsilon_{n+1}^2} I_{d} + \text{Diag} \left( \frac{1}{n+1}\sum_{k=0}^{n} \widetilde {\mathcal H}^{(k)}
\right)   \right\rbrace^{-1/2}\eqsp,
\end{equation}
where $(\varepsilon_{n})_{n\geq 1}$ is a positive sequence that plays the role of a regularization parameter.
For instance, it can be set to a regularization constant, \textit{i.e.},
$\varepsilon_{n} = \varepsilon^{-1/2}$ for some $\varepsilon > 0$. 
For each $k\ge 0$, the matrix $\widetilde{\mathcal{H}}^{(k)}$ is a $d\times d$ diagonal matrix, with entries defined for $i = 1, \ldots, d$ by
\begin{equation*}
\widetilde {\mathcal H}^{(k)}_{ii} = \min\left\lbrace \left(e_i^{\top}\widehat{H}_{\theta_{k}}^{(K_{k+1})}\left(\bX^{(k+1)}\right)e_{i} \right)^{2}, M_{k}^{2} \right\rbrace\eqsp.
\end{equation*}
The sequence $(M_n)_{n\ge 0}$ is user-specified and serves as a clipping threshold. In practice, this clipping improves numerical stability by ensuring that the smallest eigenvalue of $A_n$ does not become excessively small.
Let us define the sequence $(\underline{\varepsilon}_n)_{n\geq 0}$ by
\begin{equation*}
\underline{\varepsilon}_0 = \varepsilon_1,
\quad \underline{\varepsilon}_n = \left(\varepsilon_{n+1}^{-2} + \sup_{0 \leq i \leq n-1} M_i^2\right)^{-1/2}, \quad n \geq 1\eqsp.
\end{equation*}
This sequence plays a similar role to that of the sequence $(\underline{\lambda}_n)_{n\geq 1}$ in Assumption \ref{A-eigen}.
However, rather than providing a lower bound on the smallest eigenvalue of $A_n$, it yields a lower bound on the smallest eigenvalue of a $\mathcal{F}_n$-measurable auxiliary matrix, which is central to obtain the convergence bounds for the randomized variant of MLMC-Adagrad (see Appendix \ref{sec:app:proof-adagrad} for details).

\begin{theorem}\label{theo::adagrad}
Assume that A\ref{A-smooth} -- A\ref{A-moment-3} hold. 
For any $N \geq 1$, let $R$ be as defined in \eqref{eq:def-r} with $(\lambda_n)_{n\geq 1} = (\underline{\varepsilon}_{n-1})_{n \geq 1}$.
Then,
\begin{multline*}
\mathbb{E}\left[ 
\left\lVert \nabla V (\theta_{R})   \right\rVert^{2} 
\right] \leq 
\frac{2}{\varpi_N}\Bigg(
\mathbb{E}\left[V(\theta_0)\right] - V(\theta^{*}) + c_1G^2\sum_{n = 1}^{N+1} \gamma_{n}\frac{\varepsilon_{n}}{T_{n}}
+c_3 G^4\sum_{n = 1}^{N+1} \gamma_{n}\frac{\varepsilon_{n}^3}{n}  \sqrt{T_{n}}
\\ + c_2G^2\frac{L}{2}\sum_{n = 1}^{N+1} \gamma_{n}^2 \varepsilon_{n}^2\log(T_{n})
\Bigg)\eqsp.
\end{multline*}
\end{theorem}

\begin{corollary}
\label{cor:adagrad}
If assumptions of Theorem \ref{theo::adagrad} hold, and if for all $n \geq 1$,
\begin{equation*}
\gamma_n \propto n^{-\gamma},
\quad M_{n-1} \propto n^{M},
\quad \varepsilon_{n} \propto n^{\overline{\varepsilon}},
\quad T_n \propto n^{\alpha},
\end{equation*}
with $\alpha > 0$, and $\gamma$, $M$, $\overline{\varepsilon}$ non-negative constants such that $\gamma + M < 1 + \overline{\varepsilon}$, then
\[
\mathbb{E}\left[ 
\left\lVert \nabla V (\theta_{R})   \right\rVert^{2} 
\right] = \mathcal{O}\bigg( 
\frac{1}{N^{1+\overline{\varepsilon}-\gamma-M}} \Big\lbrace
1 + \Psi_N(\alpha + \gamma - \overline{\varepsilon})
+ \Psi_N\left(1 + \gamma -3\overline{\varepsilon} - \alpha/2\right)
  + \Phi_N(2\gamma - 2\overline{\varepsilon})
  \Big\rbrace
  \bigg)\eqsp,
\]
where for all $\eta \in \mathbb{R}$, $\Psi_N(\eta)$ and $\Phi_N(\eta)$ are as defined by \eqref{def::psi} and \eqref{def::phi}.
\end{corollary}

Following Corollary \ref{cor:adagrad}, by choosing $\gamma = 1/2$, $M = \overline{\varepsilon} = 0$, and $1/2 \leq \alpha \leq 1$, we get
\begin{equation*}
\mathbb{E} \left[ \left\Vert \nabla V (\theta_{R}) \right\rVert^{2} \right] = \mathcal{O} \left( \frac{(\log N)^{2}}{\sqrt{N}} \right)\eqsp.
\end{equation*}

\paragraph{AMSGrad. } An alternative to Adagrad is AMSGrad \citep{reddi2019convergence}, an exponential moving-average variant of Adam. Its key distinguishing feature is that, unlike Adam, the associated sequence of preconditioning matrices is non-increasing.
We introduce an MLMC variant of AMSGrad in Algorithm~\ref{algo::MLMC-AMSGrad}.
The parameters $\rho_1,\rho_2 \in [0,1)$ control the exponential forgetting of past gradient information. The parameter $\delta>0$ ensures a uniform upper bound on the eigenvalues of the sequence
$\left(A_n\right)_{n \geq 0}$ while the non-decreasing sequence $(\varepsilon_{n})_{n \geq 1}$ provides
control over the smallest eigenvalue.
Finally, the auxiliary matrix $\widehat{W}_n$ enforces the monotonicity of the preconditioning matrices,
thereby guaranteeing that the sequence $(A_n)_{n\ge0}$ is non-increasing.

\begin{algorithm}[tb]
\caption{MLMC--AMSGrad}
\label{algo::MLMC-AMSGrad}
\begin{algorithmic}
\State {\bfseries Input:} $\theta_0 \in \mathbb{R}^d$, $N \in \mathbb{N}_0$,
$\rho_1, \rho_2 \in [0,1)$, $\delta > 0$, positive non-increasing sequence $\left( \gamma_n \right)_{n \ge 1}$, positive non-decreasing sequences $\left( T_{n} \right)_{n \geq 1}$ and $\left( \varepsilon_{n} \right)_{n \geq 1}$
\State Set $m_{-1} = 0$, $W_{-1} = 0$, $\widehat{W}_{-1} = 0$
\For{$n = 0$ {\bfseries to} $N-1$}
\State Sample $K_{n+1} \sim \mathcal{G}(1/2)$ 
\State Sample a Markov chain $\bX^{(n+1)}$ with transition kernel $P_{\theta_{n}}$
\State Set $m_{n} = \rho_1 m_{n-1} + (1 - \rho_1)\Hmlmc$
\State Set
\begin{align*}
W_{n} & = \rho_2 W_{n-1} + (1 - \rho_2) \min\Big\{\varepsilon_{n+1} I_d , 
 \\ 
 & \hspace{-.0em} \text{Diag}\Big(
 \widehat{H}_{\theta_n}^{(K_{n+1})}\big(\bX^{(n+1)}\big)
 \widehat{H}_{\theta_n}^{(K_{n+1})}\big(\bX^{(n+1)}\big)^{\top}\Big)
    \Big\}
\\
\widehat{W}_{n} & =  \max\left\{\widehat{W}_{n-1},\, W_{n}\right\}
\\
A_{n} & = \left(\delta I_d + \widehat{W}_{n}\right)^{-\nicefrac{1}{2}}
\end{align*}
\State Set $\theta_{n+1} = \theta_{n} - \gamma_{n+1} A_{n} m_{n}$
\EndFor
\end{algorithmic}
\end{algorithm}
Define the sequence $(\underline{\varepsilon}_n)_{n\geq 0}$ as $\underline{\varepsilon}_0 = 0$ and for all $n\geq 1$,
\begin{equation*}
\underline{\varepsilon}_n = \frac{1}{\sqrt{\delta + \varepsilon_n(1 - \rho_2^n)}}\eqsp.
\end{equation*}

\begin{theorem}\label{theo::amsgrad}
Assume that A\ref{A-smooth} -- A\ref{A-bias-var} hold. Suppose that the sequences $(T_{n})_{n\geq 1}$ and $(\gamma_{n})_{n\geq 1}$ are respectively non-decreasing and non-increasing. 
For any $N \geq 1$, let $R$ be as defined in \eqref{eq:def-r} with $(\lambda_n)_{n\geq 1} = (\underline{\varepsilon}_{n-1})_{n \geq 1}$.
Then,
\begin{multline*}
\mathbb{E}\left[\left\lVert \nabla V(\theta_R)\right\rVert^2\right]
\leq 
\frac{1}{\varpi_N} \Bigg\{
\mathbb{E}\left[V(\theta_{0})\right] - V \left(\theta^{*}\right)
+ b_0 
+ b_1 \sum_{n = 1}^N\frac{\gamma_{n+1}}{T_{n+1}}
+ b_2 \sum_{n = 1}^{N}\gamma_{n+1}^{2}\log T_{n+1}
\\
+ b_3
\sum_{n = 1}^N
 (1 - \rho_1^n) \gamma_n^2 \log T_{n}
+ b_4 \sum_{n = 1}^N  \left(1 - \frac{\gamma_{n+1}^2}{\gamma_n^2}\right)\Bigg\}\eqsp,
\end{multline*}
with $b_0$, $b_1$, $b_2$, $b_3$ and $b_4$ constants given in Appendix \ref{sec:app-proof-amsgrad}.
\end{theorem}
\begin{corollary}
\label{cor:ams-grad}
If assumptions of Theorem \ref{theo::amsgrad} hold, and if for all $n \geq 1$,
\begin{equation*}
\gamma_n \propto n^{-\gamma},
\quad \varepsilon_{n} \propto n^{\overline{\varepsilon}},
\quad T_n \propto n^{\alpha},
\end{equation*}
with $\alpha > 0$, and $\gamma$, $\overline{\varepsilon}$ non-negative constants such that $2\gamma + \overline{\varepsilon} < 2$, then
\begin{multline*}
\mathbb{E}\left[ 
\left\lVert \nabla V (\theta_{R})   \right\rVert^{2} 
\right] = \mathcal{O}\bigg( 
\frac{1}{N^{1-\gamma - \nicefrac{\overline{\varepsilon}}{2}}} \Big\lbrace
1 + \Psi_N(\alpha + \gamma)
  + \Phi_N(2\gamma) + \gamma\log N
  \Big\rbrace
  \bigg)\eqsp,
\end{multline*}
where for all $\eta \in \mathbb{R}$, $\Psi_N(\eta)$ and $\Phi_N(\eta)$ are as defined by \eqref{def::psi} and \eqref{def::phi}.
\end{corollary}
Following Corollary \ref{cor:ams-grad}, and using the same setup as for Adagrad, we obtain a similar convergence rate of $\mathcal{O} \left( (\log N)^{2} / \sqrt{N} \right)$ for $1/2 \leq \alpha \leq 1$. 
In many convergence analyses, the bias and variance are assumed to be constant \citep{karimi2019non}, which is not always the case. The rates obtained in Theorems~\ref{theo::adagrad} and~\ref{theo::amsgrad} allow for a decreasing bias at the cost of an increasing variance, which is reflected in the resulting convergence rate. They also highlight the trade-off between the term arising from adaptive step sizes and the bias and variance terms. Finally, when $\alpha < 1/2$, convergence still holds, but the rate is dominated by the bias induced by the Markov chain.

\subsection{Verifying assumptions A\ref{A-bias-var} and A\ref{A-moment-3}}
\label{sec:check-A-moments}
Consider the following assumptions on the Markov chain.
\begin{hypH}
\label{hyp:stat}
For all $\theta\in\Theta$,  the Markov kernel $P_{\theta}$ is irreducible and aperiodic, with a unique invariant
probability $\pi_{\theta}$.
\end{hypH}


\begin{hypH}
\label{hyp:geomerg}
For all $\theta\in\Theta$, $P_\theta$ is geometrically regular, i.e. there exists $\delta>1 $ and a set $C$ such that 
\[
\mathrm{sup}_{x\in C}\mathbb{E}_x\left[\sum_{k=0}^{\sigma_C-1}\delta^k\right]<\infty\eqsp,
\]
where $\sigma_C = \mathrm{inf}_{k\geq 1}\{X_k\in C\}$.
\end{hypH}
Following Theorem 15.1.3 of \citet{douc2018markov}, assumption H\ref{hyp:geomerg} is equivalent to a Foster-Lyapunov drift condition. This assumption is crucial to obtain sub-Gaussian concentration inequalities and controls on the expectations of functionals of the Markov chain with different initializations, see Lemma~\ref{lemma:mcdiarmid} and Lemma~\ref{lemma:diff-measure}. 

\begin{proposition}\label{prop1}
Assume that A\ref{A-bounded} and H\ref{hyp:stat}--H\ref{hyp:geomerg} hold. Then, Assumptions A\ref{A-bias-var} and A\ref{A-moment-3} hold with constants
$c_1,c_2,c_3>0$ given in Appendix~\ref{sec:control:MC}.
\end{proposition}


\paragraph{Random Walk Metropolis-Hastings algorithm (RWMH). }

Let $n\geq 1$ be the current iteration and $\pi_{\theta_n}$ be the associated target distribution on $\mathbb{R}^d$, known up to a normalizing constant. The kernel $P_{\theta_n}$ is built as follows.
Given a current state $X_k$, the Random-Walk Metropolis--Hastings algorithm generates a proposal $Y = X_k + Z$, $Z \sim q$, 
where $q$ is a centred proposal density. The proposal is accepted with probability $\alpha_{\theta_n}(x,y) = 1 \wedge \pi_{\theta_n}(y)/\pi_{\theta_n}(x)$
and the chain is updated as $X_{k+1} = Y$ with probability $\alpha_{\theta_n}(X_k,Y)$ and $X_{k+1} = X_k$ otherwise.  
We denote by $P_{\theta_n}$ its Markov transition kernel which is $\pi_{\theta_n}$ invariant by construction. In the case where $\pi_{\theta_n}$ and $q$ are positive and continuous,  by \citet{mengersen1996rates}, H\ref{hyp:stat} is satisfied. Under additional assumptions, H\ref{hyp:geomerg} is satisfied. Assume that $\pi_{\theta_n}$ is log-concave in the tails and that the proposal density $q$ is symmetric positive and continuous, then, by \citet{mengersen1996rates}, the RWMH chain is geometrically
ergodic in total variation: there exist  $0<\alpha_{\theta_n}<1$
and a Lyapunov function $V_{\theta_n}:\mathbb{R}^d\to [1,\infty)$ such that $\tv(\delta_x P^k_{\theta_n},\pi_{\theta_n})\le V_{\theta_n}(x)\,\alpha_{\theta_n}^k$, for all $x\in\mathbb{R}^d$ and $k\geq 1$. Moreover, the RWMH algorithm cannot be uniformly ergodic on $\mathbb{R}^d$ 
unless the state space is compact.

\paragraph{Metropolis-Adjusted Langevin Algorithm (MALA). }
In this context, we write $\pi_{\theta_n}$ on $\mathbb{R}^d$ as $\pi_{\theta_n}(x) = Z_{\theta_n}^{-1} \mathrm{e}^{-U_{\theta_n}(x)}$ 
where $U_{\theta_n}:\mathbb{R}^d \to \mathbb{R}$ is a differentiable potential.
Given $X_k$, MALA generates a proposal
    $Y = X_k - h \nabla U_{\theta_n}(X_k) + \sqrt{2h}\,Z$, 
    $Z \sim \mathcal N(0,\mathrm{I}_d)$, and accepts it with probability $\alpha_{\theta_n}(X_k,Y)$ where
\begin{equation*}
    \alpha_{\theta_n}(x,y)
    = 1 \wedge 
    \frac{\pi_{\theta_n}(y)\,q_{\theta_n,h}(y,x)}{\pi_{\theta_n}(x)\,q_{\theta_n,h}(x,y)}\eqsp,
\end{equation*}
where $q_{\theta_n,h}$ is the Gaussian proposal density associated with the Langevin step.
The proposal is accepted with probability $\alpha_{\theta_n}(x,y) = 1 \wedge \pi_{\theta_n}(y)/\pi_{\theta_n}(x)$
and the chain is updated as $X_{k+1} = Y$ with probability $\alpha_{\theta_n}(X_k,Y)$ and $X_{k+1} = X_k$ otherwise. 
We denote by $P_{\theta_n}$ its Markov transition kernel. Following \citet{roberts1996exponential}, H\ref{hyp:geomerg} geometric ergodicity can be obtained under suitable smoothness assumptions on $U_{\theta_n}$. 

\section{Application of MLMC to Importance Weighted Autoencoders}
\label{sec:MLMC:IWAE}
Consider a framework where we have access to a dataset $\{Y_i\}_{1\leq i \leq n}$ of \emph{i.i.d.} random variables with unknown distribution $\pi_{{\rm data}}$. A parametric generative model $y\mapsto p_{\theta}(y)$ where $\theta$ is an unknown parameter can be used to estimate  $y\mapsto \pi_{{\rm data}}(y)$. In latent generative models,  $p_{\theta}(y)$ is the marginal of $(y,z)\mapsto p_{\theta}(y, z)$, where $y$ represents the observation and $z$ the latent variable. By Fisher's identity,
\begin{equation*} 
\nabla_{\theta} \log p_{\theta}(y)=\int \nabla_{\theta} \log p_{\theta}(y, z) p_{\theta}(z \mid y) \mathrm{d} z\eqsp.
\end{equation*}
In most cases, the conditional density $z\mapsto p_{\theta}(z \mid y)$ is intractable.  Variational Autoencoders \citep{kingma2013auto} introduce an additional parameter $\phi$ and a family of variational distributions $z\mapsto q_{\phi}(z \mid y)$ to approximate the true posterior distribution and the unknown
parameters are estimated by maximizing $(\theta,\phi) \mapsto \mathbb{E}_{\pi_{{\rm data}}}[\mathcal{L}_{\text{ELBO}}(\theta, \phi; Y)]$, referred to as the Evidence Lower Bound (ELBO), where 
\[
\mathcal{L}_{\text{ELBO}}(\theta, \phi; y) = 
\mathbb{E}_{q_{\phi}(\cdot\mid y)} \left[ \log \frac{p_{\theta}(y, Z)}{q_{\phi}(Z \mid y)} \right] \leq \log p_{\theta}(y) \eqsp.
\]
Importance Weighted Autoencoders \citep[IWAE,][]{burda2015importance} provide a tighter lower bound of the log-likelihood.
The IWAE objective is given by:
\begin{equation*}
\mathcal{L}^{\text{IWAE}}_k(\theta, \phi; y) = \mathbb{E}_{q^{\otimes k}_{\phi}(\cdot\mid y)} \left[ \log \frac{1}{k} \sum_{\ell=1}^{k} \frac{p_{\theta}(y, Z_{\ell})}{q_{\phi}(Z_{\ell} \mid y)} \right],
\end{equation*}
where $k$ corresponds to the number of samples drawn from the variational posterior distribution. An interesting feature of this approach is that increasing the number of samples tightens the bound:
\[
\mathcal{L}^{\text{IWAE}}_k(\theta, \phi; y) \leq \mathcal{L}^{\text{IWAE}}_{k+1}(\theta, \phi; y) \leq \log p_\theta(y)\eqsp.
\]
Notably, the estimator of the gradient of the ELBO in IWAE is a biased estimator of $\nabla_{\theta} \log p_{\theta}(y)$.  

In this setting, we can set the objective function as $V(\theta) = \mathbb{E}_{\pi_{{\rm data}}}[\log p_{\theta}(Y)]$.  The gradient of $V$ is given by:
\begin{align} \label{eq:grad_IWAE}
\nabla_{\theta} V(\theta) &= \mathbb{E}_{\pi_{{\rm data}}}[\nabla_{\theta} \log p_{\theta}(Y)]  = \mathbb{E}_{\pi_{{\rm data}}}[\mathbb{E}_{p_{\theta}(\cdot\mid Y)} \left[ \nabla_{\theta} \log p_{\theta}(Y, z) \right]]\eqsp.
\end{align}
The classical empirical estimate of the gradient of the ELBO of the IWAE objective is
\[
\widehat{\nabla}_{\theta} \mathcal{L}^{\text{IWAE}}_k(\theta, \phi; y) = \sum_{\ell=1}^{k} \frac{w_{\ell}}{\sum_{\ell=1}^{k} w_{\ell}} \nabla_{\theta} \log p_{\theta}(y, z_{\ell})\eqsp,
\]
where $y\sim \pi_{{\rm data}}$, $z_{\ell}$, $1\leq \ell \leq k$, are \emph{i.i.d.} with distribution $q_{\phi}(z_{\ell} \mid y)$ and   $w_{\ell} =
p_{\theta}(y, z_{\ell})/q_{\phi}(z_{\ell}\mid y)$ are the unnormalized importance weights.

\begin{algorithm}[tb]
   \caption{\textbf{MLMC-IWAE Gradient Estimator}}
   \label{alg:mlmc-iwae}
\begin{algorithmic}
   \State {\bfseries Input:} Maximum truncation level $T$ of the Markov chain, number of samples $k$ from the variational distribution $q_{\phi}(\cdot \mid y)$
   \State {\bfseries Initialization:} Sample $\widetilde z_{0}\sim q_{\phi}(\cdot \mid y)$, $K \sim \mathcal{G}(1/2)$ and set $t_K = 1 \vee \bigl(2^{K}\,\mathds{1}_{\{2^{K}\le T\}}\bigr)$
   \For{$p=1$ to $t_K$}
   \State Sample $J \sim \mathrm{Unif}\{1, \ldots, k\}$ and set $z_{p, J} = \widetilde z_{p-1}$
   \State Sample, for all $\ell\in\{1, \ldots, k\}\backslash\left\{J\right\}$, $z_{p, \ell}$ independently from the variational distribution $q_{\phi}(\cdot \mid y)$
   \State Compute the unnormalized importance weights: 
   \[
   w_{p, \ell} = \frac{p_{\theta}(y, z_{p, \ell})}{q_{\phi}(z_{p, \ell} \mid y)}, \quad \ell = 1, \ldots, k
   \] 
   \State Normalize importance weights:
\[
\omega_{p, \ell} = \frac{w_{p, \ell} }{\sum_{j=1}^{k} w_{j, \ell}}, \quad \ell = 1, \ldots, k
\]
   \State Select $\widetilde z_{p}$ from the set $\{z_{p,1}, \ldots, z_{p, k}\}$ according to the probability vector $(\omega_{p, 1}, \ldots, \omega_{p, k})$
   \State Compute the gradient estimator
    \[
    \Hmean[\theta][p][\mathbf{z}]  = \frac{1}{p} \sum_{i=1}^{p} \sum_{\ell=1}^{k} \omega_{i, \ell} \nabla_{\theta} \log p_{\theta}(y, z_{i, \ell})
    \]
   \EndFor
   \State {\bfseries Output:} MLMC estimator $\Hhat[\theta][K][\mathbf{z}]$ defined in \eqref{eq:mlmc-estim}
\end{algorithmic}
\end{algorithm}

The IWAE gradient estimator can be interpreted as a self-normalized importance sampling (SNIS) estimator of the true objective gradient
$ \mathbb{E}_{p_{\theta}(\cdot\mid y)} \left[ \nabla_{\theta} \log p_{\theta}(y, z) \right]$ and thus belongs to a broad class of biased estimators. Following \citet{cardoso2022br}, an appealing approach is to use a BR-SNIS procedure, which relies on iterated sampling--importance resampling to construct a bias-reduced estimator. We refer to BR-IWAE as the BR-SNIS estimator applied to IWAE.
Accordingly, the same idea can be used to propose an MLMC approach, summarized in Algorithm~\ref{alg:mlmc-iwae}. Note that the MLMC gradient estimator is computed only for $\theta$. For $\phi$, we use the standard IWAE estimator in both methods, since increasing $k$ may lead to a vanishing signal-to-noise ratio and consequently poor gradient estimates for $\phi$ \citep{surendran2024theoretical}.
Compared to BR-IWAE, MLMC-IWAE achieves the same order of bias with a smaller computational budget. In practice, $T$ can be chosen to trade off the bias of the gradient estimator against computational cost.
  


\section{Experiments}
\label{sec:exp}

In this section, we illustrate our theoretical results in the context of Importance-Weighted Autoencoders, detailed in Section \ref{sec:MLMC:IWAE}. Our source code is publicly available at\footnote{URL hidden during the review process}.

\textbf{Dataset and Model. } We conduct our experiments on the CIFAR-10 dataset \citep{krizhevsky2009learning} and use a convolutional neural network architecture for both the encoder and the decoder. The latent space dimension is set to 100. We compare BR-IWAE \citep{cardoso2022br} with the proposed MLMC-IWAE (Algorithm~\ref{alg:mlmc-iwae}), using $k=5$ and $T_n = n^{1/2}$. Training is conducted using AMSGrad with a decaying learning rate $\gamma_{n} = C_{\gamma}/\sqrt{n}$, where $C_{\gamma}=0.001$, over 200 epochs. The momentum parameters are set to $\rho_1 = 0.9$ and $\rho_2 = 0.999$. Note that although all figures are plotted with respect to epochs, $N$ denotes the number of gradient updates. We follow the same empirical protocol as in \citet{surendran2024non} and benchmark our MLMC approach against BR-IWAE, which is reported as the best-performing method among the considered variants. This choice allows a fair comparison with a strong baseline while avoiding additional unnecessary computational costs. Since the main purpose of the empirical study is to highlight convergence rates, this setup also supports the claim that MLMC alternatives are broadly applicable to a large class of MCMC algorithms.

\textbf{Evaluation metric. }
To illustrate the convergence bound, we report the squared gradient norm $\|\nabla V(\theta_N)\|^2$ along training.
To evaluate sample quality, we report the negative log-likelihood (NLL) on the test set, estimated using the IWAE bound with $1000$ importance samples.

\begin{figure}[t]
  \begin{center}
    \centerline{\includegraphics[width=\columnwidth]{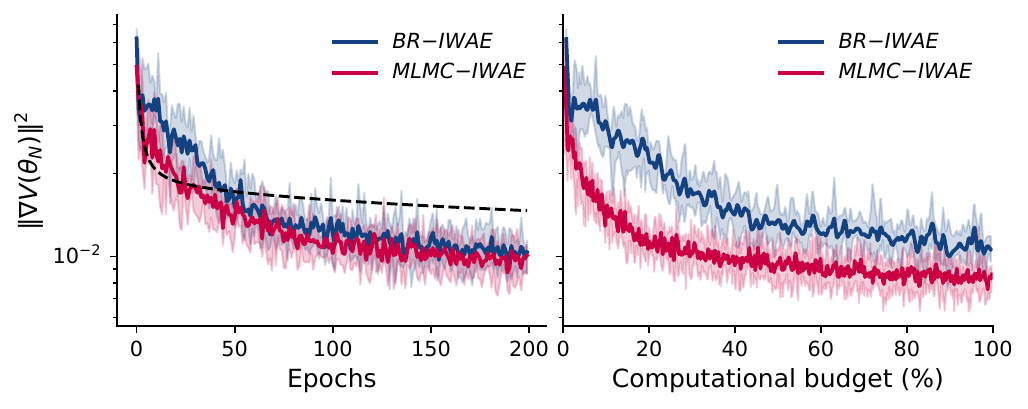}}
    \caption{
      Squared gradient norm $\|\nabla V(\theta_N)\|^2$ for BR-IWAE and MLMC-IWAE trained with AMSGrad on CIFAR-10, shown as a function of epochs (left) and computational budget (right). Both plots use the same scale and are displayed on a logarithmic scale for improved readability. Bold lines represent the mean over 5 independent runs.
    }
    \label{fig:grad_norm_amsgrad}
  \end{center}
\end{figure}

\begin{figure}[t]
  \begin{center}
    \centerline{\includegraphics[width=\columnwidth]{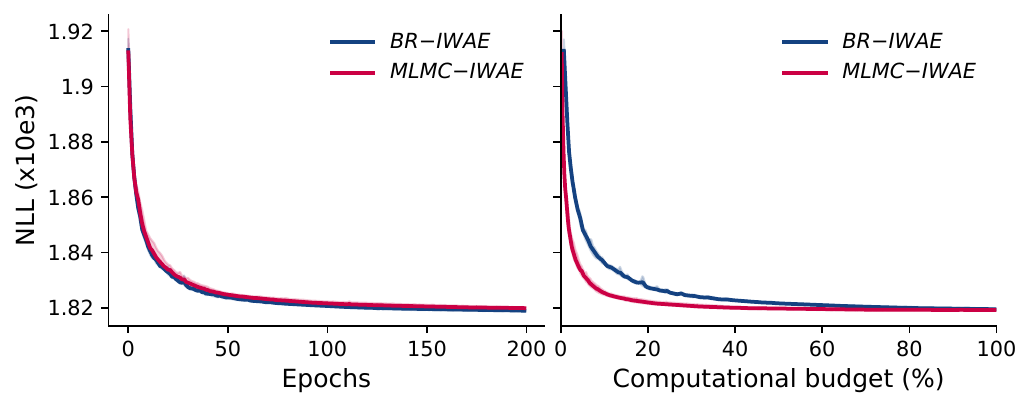}}
    \caption{
      Negative log-likelihood for BR-IWAE and MLMC-IWAE trained with AMSGrad on CIFAR-10, shown as a function of epochs (left) and computational budget (right). Bold lines represent the mean over 5 independent runs.
    }
    \label{fig:test_nll_amsgrad}
  \end{center}
\end{figure}

Figure~\ref{fig:grad_norm_amsgrad} and Figure~\ref{fig:test_nll_amsgrad} report the squared gradient norm and the test negative log-likelihood, respectively, as functions of the number of epochs and the computational budget. We observe similar convergence rates for both BR-IWAE and MLMC-IWAE in terms of iterations for the squared gradient norm. The dashed curve corresponds to the expected rate $\mathcal{O}\!( (\log N)^{2}/\sqrt{N} )$, and the empirical results match the convergence rate of Theorem~\ref{theo::amsgrad}.

For the negative log-likelihood, BR-IWAE performs slightly better, since it uses more samples to estimate the gradient. However, when these quantities are plotted against the computational budget, MLMC-IWAE converges faster for both the squared gradient norm and the test negative log-likelihood.
This conclusion remains consistent with Adagrad; see Appendix~\ref{sec:add_exp}.
Overall, these results suggest that MLMC-IWAE achieves comparable per-iteration performance while requiring fewer effective samples per unit of computation. In practice, choosing the maximal truncation level $T_n$ appropriately is crucial to balance bias reduction (and thus the achievable convergence rate) against computational budget.

\section{Discussion}
In this paper, we introduced a novel multilevel Markov chain Monte Carlo framework for adaptive stochastic gradient methods and proposed multilevel variants of widely used adaptive algorithms, including AdaGrad and AMSGrad, in settings where gradients are estimated via Markov chain–based procedures. We established explicit convergence results under assumptions controlling the bias and the second and third moments of the gradient estimator, and we illustrated these results in a challenging setting through comparisons with state-of-the-art Markov chain–based estimators. The proposed methodology is broadly applicable to a wide range of MCMC algorithms. The main theoretical limitation of our approach is the lack of explicit ergodicity constants, which prevents fully quantified convergence bounds, and the dependence of objective-related constants on the specific models and deep architectures considered. Although our experiments highlight that our MLMC-based method converges faster than the alternatives, data-driven tuning of the hyperparameters remains an open problem.


\section*{Impact Statement}
This paper presents work whose goal is to advance the field of Machine
Learning. There are many potential societal consequences of our work, none
which we feel must be specifically highlighted here.

\bibliographystyle{apalike}

\bibliography{MarkovMLMC.bib}

\appendix

\section{Controls of the MLMC gradient estimates}
\label{sec:control:MC}

Let $\mu$ be a probability measure on $\mathcal{K} = \{1, 2, \ldots, m\} \subseteq\mathbb{N}$, whose density with respect to the counting measure is also denoted $\mu$, such that $k\mapsto\mu(k)$ is non-increasing. We set
\begin{equation*}
    \tau(0) = \max\left(1, \frac{1}{2\mu(1)}\right), \quad \tau(k) = \frac{1}{\mu(k)}, \quad k \in \mathcal{K}\eqsp.
\end{equation*}
The arbitrary choice of $\tau(0)$ is meant to give a non-decreasing sequence $\tau(k)$.
Let $(K_n)_{n\geq 1}$ be a sequence of random variables independent and identically distributed according to $\mu$, and independent of $\bX^{(n)}$. Let $(\trunc_n)_{n\geq 1}$ be a positive sequence with $\trunc_1 \geq \lfloor\tau(1) \rfloor$, and define
\begin{equation}
\label{eq:invtrunc}
\invtrunc_n = \max\{k\in\mathcal{K}\,:\, \lfloor \tau(k) \rfloor \leq \trunc_n\}\eqsp.
\end{equation}
For all real value $r\geq 1$ and all $\bx\in\Xset^{\mathbb{N}}$, we define
\begin{equation*}
\Hmean[\theta][r][\bx] = \frac{1}{\lfloor r \rfloor}\sum_{i = 1}^{\lfloor r \rfloor} H_{\theta}(x_i)\eqsp.
\end{equation*}
At iteration $n \geq 0$, given a parameter $\theta_{n} \in \Theta$ and a Markov chain $\bX^{(n+1)}$ with Markov kernel $P_{\theta_{n}}$, we define
\begin{align*}
\Hhat[\theta_{n}][k]
& = H_{\theta_{n}}\left(X_1^{(n+1)}\right) +  \tau(k) \left\{ 
\Hmean[\theta_{n}][\tau(k)] - \Hmean[\theta_{n}][\tau(k-1)]
\right\}\mathds{1}_{\lfloor \tau(k) \rfloor \leq \trunc_{n+1}}
\\
& = H_{\theta_{n}}\left(X_1^{(n+1)}\right) +  \tau(k) \left\{ 
\Hmean[\theta_{n}][\tau(k)] - \Hmean[\theta_{n}][\tau(k-1)]
\right\}\mathds{1}_{k \leq \invtrunc_{n+1}}\eqsp.
\end{align*}

\begin{lemma}
\label{lemma:trunc}
For all $n\geq 0$,
\begin{equation*}
    \mathbb{E}\left[
    \Hmlmc \mid \mathcal{F}_{n}
    \right] 
    = \mathbb{E} \left[
    \Hmean \mid \mathcal{F}_{n}  
    \right]\eqsp.
\end{equation*}
\end{lemma}

\begin{proof}
Let $n \geq 0$ and let $\mathcal{F}_{n} \vee \bX^{(n+1)} = \sigma(\mathcal{F}_{n} \cup \sigma(\bX^{(n+1)}))$ be the smallest $\sigma$-algebra containing both $\mathcal{F}_{n}$ and the $\sigma$-algebra generated by the random variable $\bX^{(n+1)}$.
Applying the tower property  yields
\begin{align*}
    \mathbb{E}\left[
    \Hmlmc \mid \mathcal{F}_{n}
    \right] 
    & = 
    \mathbb{E}\left[
    \mathbb{E}\left[ \Hmlmc \mid \mathcal{F}_{n} \vee \bX^{(n+1)} \right] \mid \mathcal{F}_{n}
    \right]\eqsp.
\end{align*}
Using that $H_{\theta_{n}}\left(X_1^{(n+1)}\right)= \Hmean[\theta_{n}][1][\bX^{(n+1)}]$ is measurable with respect to $\mathcal{F}_{n} \vee \bX^{(n+1)}$, and that $K_{n+1}$ is independent of $\mathcal{F}_{n} \vee \bX^{(n+1)}$ we get
\begin{align*}
    & \mathbb{E}\left[
    \Hmlmc \mid \mathcal{F}_{n}
    \right] 
    \\
    & \qquad\qquad =
    \mathbb{E}\left[
    \Hmean[\theta_{n}][1][\bX^{(n+1)}]
    +  \sum_{k = 1}^{\invtrunc_{n+1}} \mathbb{P}[K_{n+1} = k] \tau(k)\left\{ 
    \Hmean[\theta_{n}][\tau(k)][\bX^{(n+1)}] - \Hmean[\theta_{n}][\tau(k - 1)][\bX^{(n+1)}]
    \right\} \mid \mathcal{F}_{n}
    \right]\eqsp.
\end{align*}
Since for $1 \leq k \leq \invtrunc_{n+1} \leq m$, $\mathbb{P}[K_{n+1} = k] = \mu(k) = 1/\tau(k)$, the latter yields by a telescopic argument
\begin{equation*}
\mathbb{E}\left[\Hmlmc \mid \mathcal{F}_{n-1}\right] = \mathbb{E}\left[\Hmean \mid \mathcal{F}_{n-1}\right]\eqsp.
\end{equation*}
\end{proof}

\begin{lemma} 
\label{lemma:coord}
Assume that A\ref{A-bounded} and H\ref{hyp:stat}--H\ref{hyp:geomerg} hold, and for all $n\geq 1$, $\bX^{(n)}_1 = x^{(n)}_1\in C$. Then, there exist $\alpha \in [0,1)$ and $\zeta<\infty$ such that for all $n \geq 0$, $k \in \mathcal{K}$, and $1\leq j \leq d$, 
\begin{equation*}
\left\vert \mathbb{E} \left[
     e_j^\top\,\Hmean[\theta_{n}][\tau(k)] \mid \mathcal{F}_{n}  
    \right] -   \mathbb{E}_{\pi_{\theta_{n}}} \left[e_j^\top\, \Hmean[\theta_{n}][\tau(k)] \mid \mathcal{F}_{n}\right]\right\rvert \leq \frac{2\zeta G}{(1 - \alpha)\lfloor \tau(k) \rfloor}\eqsp.
\end{equation*}
Moreover, there exists a positive constant $\beta$, such that for all $\varepsilon > 0$,
\begin{equation*}
    \mathbb{P}\left[
    \left\lvert
    e_j^\top\, \Hmean[\theta_{n}][\tau(k)] -
    \mathbb{E}\left[  e_j^\top\, \Hmean[\theta_{n}][\tau(k)] \mid \mathcal{F}_{n}  \right]
    \right\rvert > \varepsilon \mid \mathcal{F}_{n}
    \right] 
    \leq 2\exp\left(-\frac{\beta \varepsilon^2 \lfloor \tau(k) \rfloor}{4G^2}\right)\eqsp.
\end{equation*}
\end{lemma}

\begin{proof}
Let $n\geq 0$. For all $\bx, \by \in \Xset^{\lfloor \tau(k) \rfloor}$ and $1\leq j\leq d$,
\begin{equation*}
    \left\lvert
        e_j^\top\,\Hmean[\theta_{n}][\tau(k)][\bx] - e_j^\top\,\Hmean[\theta_{n}][\tau(k)][\by]
    \right\rvert
    \leq \frac{1}{\lfloor \tau(k) \rfloor}\sum_{i = 1}^{\lfloor \tau(k) \rfloor} \left\lvert e_j^\top H_{\theta_{n}}\left(x_i\right) - e_j^\top H_{\theta_{n}}\left(y_i\right) \right\rvert\eqsp.
\end{equation*}
It follows from Assumption A\ref{A-bounded} that $\bx\mapsto e_j^\top\,\Hmean[\theta_{n-1}][\tau(k)][\bx]$ satisfies the bounded difference condition \eqref{eq:bd-cond} with constants $\kappa_i = 2G/\lfloor \tau(k) \rfloor$, $1\leq i \leq \lfloor \tau(k) \rfloor$. 
Therefore, since H\ref{hyp:stat}-H\ref{hyp:geomerg} hold, and $\bX^{(n+1)}_1 = x^{(n+1)}_1\in C$, by Lemma \ref{lemma:diff-measure}, there exist $\alpha \in [0,1)$ and $\zeta<\infty$ such that
\begin{align*}
\left\lvert \mathbb{E} \left[
     e_j^\top\,\Hmean[\theta_{n}][\tau(k)] \mid \mathcal{F}_{n}  
    \right] -   \mathbb{E}_{\pi_{\theta_{n}}} \left[e_j^\top\, \Hmean[\theta_{n}][\tau(k)] \mid \mathcal{F}_{n}\right]\right\rvert 
    & \leq \zeta \sum_{i=1}^{\lfloor \tau(k) \rfloor} \frac{2G}{\lfloor \tau(k) \rfloor} \alpha^{i-1}
    \leq \frac{2\zeta G}{(1 - \alpha)\lfloor \tau(k) \rfloor}\eqsp.
\end{align*}
Under Assumption H\ref{hyp:stat}, we can also apply Lemma \ref{lemma:mcdiarmid}: there exists a positive constant $\beta$ such that for all coordinates $1 \leq j \leq d$, all $k\in\mathcal{K}$ and all $\varepsilon > 0$,
\begin{equation*}
    \mathbb{P}\left[
    \left\lvert
    e_j^\top\, \Hmean[\theta_{n}][\tau(k)] -
    \mathbb{E}\left[  e_j^\top\, \Hmean[\theta_{n}][\tau(k)] \mid \mathcal{F}_{n-1}  \right]
    \right\rvert > \varepsilon \mid \mathcal{F}_{n}
    \right] 
    \leq 2\exp\left(-\frac{\beta \varepsilon^2 \lfloor \tau(k) \rfloor}{4G^2}\right)\eqsp.
\end{equation*}
\end{proof}

\begin{lemma}
\label{lemma:bias-gen}
Assume that A\ref{A-bounded} and H\ref{hyp:stat}--H\ref{hyp:geomerg} hold, and for all $n\geq 1$, $\bX^{(n)}_1 = x^{(n)}_1\in C$. Then, there exist $\alpha \in [0,1)$ and $\zeta<\infty$ such that for all $n \geq 0$,
\begin{equation*}
\left\lVert \mathbb{E} \left[
     \Hmlmc \mid \mathcal{F}_{n}  
    \right] -   h\left(\theta_{n}\right) \right\rVert \leq  \frac{2\zeta G\sqrt{d}}{(1 - \alpha)\lfloor \tau(\invtrunc_{n+1}) \rfloor}\eqsp.
\end{equation*}
\end{lemma}

\begin{proof}
Let $n \geq 0$. By Lemma \ref{lemma:trunc}, we have
\begin{equation*}
    \left\lVert \mathbb{E}\left[
     \Hmlmc \mid \mathcal{F}_{n}  
    \right] -   h\left(\theta_{n}\right) \right\rVert^2
    =
    \left\lVert \mathbb{E}\left[
     \Hmean \mid \mathcal{F}_{n}  
    \right] -   h\left(\theta_{n}\right) \right\rVert^2\eqsp.
\end{equation*}
Moreover, it follows from Assumption H\ref{hyp:stat} that
\begin{equation*}
    h(\theta_{n}) = \mathbb{E}_{\pi_{\theta_{n}}}\left[\Hmean \mid \mathcal{F}_{n}\right]\eqsp.
\end{equation*}
Then
\begin{multline}
\label{eq:diff-coord}
\left\lVert \mathbb{E}\left[
     \Hmlmc \mid \mathcal{F}_{n}  
    \right] -   h\left(\theta_{n}\right) \right\rVert^2 \\= \sum_{j=1}^d \left\lvert \mathbb{E} \left[
     e_j^\top\,\Hmean \mid \mathcal{F}_{n}  
    \right] -   \mathbb{E}_{\pi_{\theta_{n}}} \left[e_j^\top\,\Hmean \mid \mathcal{F}_{n}\right]\right\rvert^2\eqsp,
\end{multline}
Applying Lemma \ref{lemma:coord} to each terms of Equation \eqref{eq:diff-coord} yields
\begin{equation*}
\left\lVert \mathbb{E} \left[
     \Hmlmc \mid \mathcal{F}_{n}  
    \right] -   h\left(\theta_{n}\right) \right\rVert \leq  \frac{2\zeta G\sqrt{d}}{(1 - \alpha)\lfloor \tau(\invtrunc_{n+1}) \rfloor}\eqsp.
\end{equation*}
\end{proof}

\begin{lemma}
\label{lemma:mom-2p-gen}
Assume that A\ref{A-bounded} and H\ref{hyp:stat}--H\ref{hyp:geomerg} hold, and for all $n\geq 1$, $\bX^{(n)}_1 = x^{(n)}_1\in C$. Then, there exist a positive constant $\beta$, $\alpha \in [0,1)$ and $\zeta<\infty$ such that for all $n \geq 0$, $p \geq 2$
\begin{multline*}
\mathbb{E}\left[
\left\lVert \Hmlmc \right\rVert^{p}
\right] 
\\ 
\leq 2^{p-1} G^{p} \left\lbrace 1 + 2^{2p}
+ \frac{16^{p}}{8}p\Gamma\left(\frac{p}{2}\right)\left(\frac{d}{\beta}\right)^{\nicefrac{p}{2}} \sum_{k = 1}^{\invtrunc_n} \tau(k)^{\frac{p}{2}-1} \Delta\left(k, \frac{p}{2}\right)
+ \frac{16^{p}}{8}\left(\frac{\zeta\sqrt{d}}{1 - \alpha}\right)^p\sum_{k = 1}^{\invtrunc_n} \frac{\Delta(k, p)}{\tau(k)}
\right\rbrace\eqsp.
\end{multline*}
where
\begin{equation*}
\Delta(k, p) = \left\lbrace
\left(\frac{\tau(k)}{\lfloor \tau(k) \rfloor}\right)^p
+
\left(\frac{\tau(k)}{\lfloor \tau(k-1) \rfloor}\right)^p
\right\rbrace\eqsp.
\end{equation*}
\end{lemma}
\begin{proof}
Let $n \geq 0$, $p \geq 2$.
Note that Jensen's inequality yields
\begin{multline*}
    \mathbb{E}\left[
    \left\lVert \Hmlmc \right\rVert^{p}
    \right] 
    \leq  2^{p-1}\mathbb{E}\left[\left\lVert \Hmlmc - \mathbb{E}\left[
    \Hmlmc \mid \mathcal{F}_{n}
    \right]  \right\rVert^{p}\right] + \\
    2^{p-1}\mathbb{E}\left[\left\lVert \mathbb{E}\left[
    \Hmlmc \mid \mathcal{F}_{n}
    \right]  \right\rVert^{p}\right]\eqsp.
\end{multline*}
From Jensen's inequality and bounded assumption A\ref{A-bounded},
\begin{equation}
\label{eq:centered:mom}
    \mathbb{E}\left[\left\lVert \mathbb{E}\left[
    \Hmlmc \mid \mathcal{F}_{n}
    \right]  \right\rVert^{p}\right] 
    = \mathbb{E}\left[\left\lVert \mathbb{E} \left[
    \Hmean \mid \mathcal{F}_{n}  
    \right] \right\rVert^{p}\right] 
    \leq  G^{p}\eqsp.
\end{equation}
On the other hand, applying Lemma \ref{lemma:trunc} and the tower rule, we get
\begin{align*}
& \mathbb{E}\left[\left\lVert \Hmlmc - \mathbb{E}\left[
    \Hmlmc \mid \mathcal{F}_{n}
    \right]  \right\rVert^{p}\right]
\\ 
& \qquad = \mathbb{E}\left[\left\lVert \Hmlmc - \mathbb{E}\left[
    \Hmean \mid \mathcal{F}_{n}
    \right]  \right\rVert^{p}\right] 
\\
& \qquad = \mathbb{E}\left[\mathbb{E}\left[\left\lVert \Hmlmc - \mathbb{E}\left[
    \Hmean \mid \mathcal{F}_{n}
    \right]  \right\rVert^{p} \mid \mathcal{F}_{n} \right] \right]\eqsp.
\end{align*}
Following the same steps as for the proof of Lemma \ref{lemma:trunc}, the independence between $K_{n+1}$ and $\bX^{(n+1)}$ yields
\begin{multline*}
\mathbb{E}\left[\left\lVert \Hmlmc - \mathbb{E}\left[
    \Hmean \mid \mathcal{F}_{n}
    \right]  \right\rVert^{p} \mid \mathcal{F}_{n} \right]
\\
\qquad = \sum_{k = 1}^{m} \mathbb{P}[K_{n+1} = k] 
\mathbb{E}\left[\left\lVert \Hhat[\theta_{n}][k] - \mathbb{E}\left[
    \Hmean \mid \mathcal{F}_{n}
    \right]  \right\rVert^{p} \mid \mathcal{F}_{n} \right]\eqsp.
\end{multline*}
By definition of $\Hhat[\theta_{n}][k]$, Jensen's inequality and Assumption A\ref{A-bounded}, for any $\invtrunc_{n+1} < k \leq m$, 
\begin{align*}
\left\lVert 
\Hhat[\theta_{n}][k] - \mathbb{E}\left[\Hmean \mid \mathcal{F}_{n} \right]  
\right\rVert^{p}
& = \left\lVert 
H_{\theta_{n}}\left(X_1^{(n+1)}\right) - \mathbb{E}\left[\Hmean \mid \mathcal{F}_{n} \right]  
\right\rVert^{p}
\\
& \leq 
2^{p-1}\left\lbrace
 \left\lVert H_{\theta_{n}}\left(X_1^{(n+1)}\right) \right\rVert^{p}
+ \mathbb{E}\left[
    \left\lVert \Hmean \right\rVert^{p} \mid \mathcal{F}_{n}
    \right]  
\right\rbrace
\\
& \leq (2G)^{p}\eqsp,
\end{align*}
and for any $1 \leq k \leq \invtrunc_{n+1}$,
\begin{align*}
& \left\lVert 
\Hhat[\theta_{n}][k] - \mathbb{E}\left[\Hmean \mid \mathcal{F}_{n} \right]  
\right\rVert^{p}
\\
&   = \left\lVert 
H_{\theta_{n}}\left(X_1^{(n+1)}\right) +  \tau(k) \left\{ 
\Hmean[\theta_{n}][\tau(k)] - \Hmean[\theta_{n}][\tau(k-1)]
\right\} - \mathbb{E}\left[\Hmean \mid \mathcal{F}_{n} \right]  
\right\rVert^{p}
\\
&  \leq 
2^{p-1}  \left\lVert 
H_{\theta_{n}}\left(X_1^{(n+1)}\right) - \mathbb{E}\left[\Hmean \mid \mathcal{F}_{n} \right]  
\right\rVert^{p} 
+ 2^{p-1}\tau(k)^{p} \left\lVert 
\Hmean[\theta_{n}][\tau(k)] - \Hmean[\theta_{n}][\tau(k-1)]
\right\rVert^{p} 
\\
&   \leq 2^{p-1}(2G)^{p} + 2^{p-1}\tau(k)^{p} \left\lVert 
\Hmean[\theta_{n}][\tau(k)] - \Hmean[\theta_{n}][\tau(k-1)]
\right\rVert^{p}\eqsp.
\end{align*}
Therefore,
\begin{multline}
\label{eq:secondterm}
\mathbb{E}\left[\left\lVert \Hmlmc - \mathbb{E}\left[
    \Hmean \mid \mathcal{F}_{n}
    \right]  \right\rVert^{p} \mid \mathcal{F}_{n} \right]
\\
\leq \sum_{k = 1}^{\invtrunc_{n+1}} \mathbb{P}[K_{n+1} = k] 
\left\lbrace 2^{p-1}(2G)^{p} + 2^{p-1}\tau(k)^{p} 
\mathbb{E}\left[
\left\lVert 
\Hmean[\theta_{n}][\tau(k)] - \Hmean[\theta_{n}][\tau(k-1)]
\right\rVert^{p} \mid \mathcal{F}_{n}
\right]
\right\rbrace 
\\
+ (2G)^{p}\sum_{k = \invtrunc_{n+1} + 1}^m  \mathbb{P}[K_{n+1} = k]\eqsp.
\end{multline}

Combining Equations \eqref{eq:centered:mom} and \eqref{eq:secondterm}, we get
\begin{multline}
\label{eq:mom-2p-gen-bound}
\mathbb{E}\left[ \left\lVert \Hmlmc \right\rVert^{p} \right] 
\leq 
\frac{(2G)^{p}}{2} \left(
1 + 2^{2p - 1}\sum_{k = 1}^{\invtrunc_{n+1}} \mathbb{P}[K_{n+1} = k] + 2^{p} \sum_{k = \invtrunc_{n+1} + 1}^m   \mathbb{P}[K_{n+1} = k]
\right)
\\
+ 4^{p-1} \sum_{k = 1}^{\invtrunc_{n+1}} \tau(k)^{p-1} 
\mathbb{E}\left[\mathbb{E}\left[
\left\lVert 
\Hmean[\theta_{n}][\tau(k)] - \Hmean[\theta_{n}][\tau(k-1)]
\right\rVert^{p} \mid \mathcal{F}_{n}
\right]\right]\eqsp.
\end{multline}

The first term on the right hand side can be bounded using
\begin{align*}
1 + 2^{2p - 1}\sum_{k = 1}^{\invtrunc_{n+1}} \mathbb{P}[K_{n+1} = k] + 2^{p} \sum_{k = \invtrunc_{n+1} + 1}^m   \mathbb{P}[K_{n+1} = k]
& = 1 + 2^{p} + \left(2^{2p-1} - 2^{p}\right) \sum_{k = 1}^{\invtrunc_{n+1}} \mathbb{P}[K_{n+1} = k]
\\
& \leq 1 + 2^{2p}\eqsp. 
\end{align*}

In addition, for all $1 \leq k \leq \invtrunc_{n+1}$, since $\left\lVert 
\mathbb{E}\left[ \Hmean[\theta_{n}][\tau(k)][\bX^{(n+1)}] \mid \mathcal{F}_{n}\right] - h(\theta_{n})
\right\rVert^{p}$ is $\mathcal{F}_{n}$-measurable, the Jensen's inequality yields
\begin{equation}
\label{eq:telescopic-diff}
\begin{aligned}
& \mathbb{E}\left[
\left\lVert 
\Hmean[\theta_{n}][\tau(k)] - \Hmean[\theta_{n}][\tau(k-1)]
\right\rVert^{p}
\mid \mathcal{F}_{n} 
\right]
\\
& \qquad \leq 
4^{p-1} \mathbb{E}\left[
\left\lVert 
\Hmean[\theta_{n}][\tau(k)] - \mathbb{E}\left[ \Hmean[\theta_{n}][\tau(k)] \mid \mathcal{F}_{n}\right]
\right\rVert^{p}
\mid \mathcal{F}_{n} \right]
\\ 
& \qquad \qquad 
+ 4^{p-1}
\left\lVert 
\mathbb{E}\left[ \Hmean[\theta_{n}][\tau(k)] \mid \mathcal{F}_{n}\right] - h(\theta_{n})
\right\rVert^{p}
\\
& \qquad\qquad 
+ 4^{p-1}
\left\lVert 
\Hmean[\theta_{n}][\tau(k-1)] - \mathbb{E}\left[ \Hmean[\theta_{n}][\tau(k-1)] \mid \mathcal{F}_{n}\right]
\right\rVert^{p}
\\ 
& \qquad \qquad 
+ 4^{p-1} \mathbb{E}\left[
\left\lVert 
\mathbb{E}\left[ \Hmean[\theta_{n}][\tau(k-1)] \mid \mathcal{F}_{n}\right] - h(\theta_{n})
\right\rVert^{p}
\mid \mathcal{F}_{n} \right]\eqsp.
\end{aligned}
\end{equation}
Following the same steps as the proof of the bias yields
\begin{equation}
\label{eq:error-2p}
\left\lVert 
\mathbb{E}\left[ \Hmean[\theta_{n}][\tau(k)] \mid \mathcal{F}_{n}  \right] - h(\theta_{n})
\right\rVert^{p}
\leq
 \left(\frac{2\zeta G\sqrt{d}}{(1 - \alpha)\lfloor \tau(k) \rfloor}\right)^{p}\eqsp.
\end{equation}
Moreover, for all $k \in\mathcal{K}$, the Hölder inequality yields
\begin{align*}
    \left\lVert
    \Hmean[\theta_{n}][\tau(k)] -
    \mathbb{E}\left[ \Hmean[\theta_{n}][\tau(k)] \mid \mathcal{F}_{n} \right]
    \right\rVert^{p}
    & =
    \left\lbrace
    \sum_{j = 1}^d 
    \left\lvert
    e_j^\top\, \Hmean[\theta_{n}][\tau(k)] -
    \mathbb{E}\left[  e_j^\top\, \Hmean[\theta_{n}][\tau(k)] \mid \mathcal{F}_{n}  \right]
    \right\rvert^2
    \right\rbrace^{\nicefrac{p}{2}}
    \\
    & \leq
    d^{\frac{p}{2}-1}
    \sum_{j = 1}^d 
    \left\lvert
    e_j^\top\, \Hmean[\theta_{n}][\tau(k)] -
    \mathbb{E}\left[  e_j^\top\, \Hmean[\theta_{n}][\tau(k)] \mid \mathcal{F}_{n}  \right]
    \right\rvert^{p}\eqsp.
\end{align*}
The layer-cake representation for non-negative random variables yields
\begin{equation}
\label{eq:layer-cake}
\begin{aligned}
    & \mathbb{E}\left[
    \left\lVert
    \Hmean[\theta_{n}][\tau(k)] -
    \mathbb{E}\left[ \Hmean[\theta_{n}][\tau(k)] \mid \mathcal{F}_{n} \right]
    \right\rVert^{p} \mid \mathcal{F}_{n}
    \right]
    \\
    & \qquad\leq d^{\frac{p}{2}-1} \sum_{j=1}^d\int_{0}^\infty
    \mathbb{P}\left[
    \left\lvert
    e_j^\top\, \Hmean[\theta_{n}][\tau(k)] -
    \mathbb{E}\left[  e_j^\top\, \Hmean[\theta_{n}][\tau(k)]  \mid \mathcal{F}_{n}  \right]
    \right\rvert^{p} > \varepsilon \mid \mathcal{F}_{n}
    \right] \rmd \varepsilon\eqsp.
\end{aligned}
\end{equation}
Since H\ref{hyp:stat}--H\ref{hyp:geomerg} hold, and $\bX^{(n+1)}_1 = x^{(n+1)}_1\in C$, Lemma \ref{lemma:coord} yields there exists a positive constant $\beta$ such that for all coordinates $1 \leq j \leq d$, all $k\in\mathcal{K}$ and all $\varepsilon > 0$,
\begin{equation*}
    \mathbb{P}\left[
    \left\lvert
    e_j^\top\, \Hmean[\theta_{n}][\tau(k)] -
    \mathbb{E}\left[  e_j^\top\, \Hmean[\theta_{n}][\tau(k)]  \mid \mathcal{F}_{n} \right]
    \right\rvert > \varepsilon^{\nicefrac{1}{p}} \mid \mathcal{F}_{n}
    \right] 
    \leq 2\exp\left(-\frac{\beta \varepsilon^{\nicefrac{2}{p}} \lfloor \tau(k) \rfloor}{4G^2}\right)\eqsp,
\end{equation*}
By combining this upper bound with Equations \eqref{eq:layer-cake}, and integrating out $\varepsilon$, we get
\begin{align}
\label{eq:result-mcdiarmid}
\mathbb{E}\left[
    \left\lVert
    \Hmean[\theta_{n}][\tau(k)] -
    \mathbb{E}\left[ \Hmean[\theta_{n}][\tau(k)] \mid \mathcal{F}_{n}  \right]
    \right\rVert^{p} \mid \mathcal{F}_{n}
    \right]
    & \leq d^{\frac{p}{2}-1} \sum_{j=1}^d\int_{0}^{\infty} 2\exp\left(-\frac{\beta \varepsilon^{\nicefrac{2}{p}} \lfloor \tau(k) \rfloor}{4G^2}\right) \rmd \varepsilon \nonumber
    \\
    & \leq p\Gamma\left(\frac{p}{2}\right) \left(\frac{4dG^2}{\beta \lfloor \tau(k) \rfloor}\right)^{\nicefrac{p}{2}}\eqsp,
\end{align}
where $\Gamma$ is the gamma function $\Gamma:p \mapsto \int_0^\infty u^{p-1}\rme^{-u}\rmd u$.
By plugging Equations \eqref{eq:error-2p} and \eqref{eq:result-mcdiarmid} into \eqref{eq:telescopic-diff}, we get
\begin{equation}
\label{eq:bound-telescopic-diff}
\begin{aligned}
& \sum_{k = 1}^{\invtrunc_{n+1}} \tau(k)^{p-1} 
\mathbb{E}\left[
\left\lVert 
\Hmean[\theta_{n}][\tau(k)] - \Hmean[\theta_{n}][\tau(k-1)]
\right\rVert^{p} \mid \mathcal{F}_{n}
\right]
\\
& \qquad \leq
2^{3p-2}  d^{\nicefrac{p}{2}}
\left(\frac{\zeta}{1 - \alpha}\right)^{p} G^{p} \sum_{k = 1}^{\invtrunc_{n+1}} \frac{1}{\tau(k)}
\left\lbrace
\left(\frac{\tau(k)}{\lfloor \tau(k) \rfloor}\right)^p
+
\left(\frac{\tau(k)}{\lfloor \tau(k-1) \rfloor}\right)^p
\right\rbrace
\\
& \qquad\qquad 
+ 2^{3p-2} p\Gamma\left(\frac{p}{2}\right) \left(\frac{d}{\beta}\right)^{\nicefrac{p}{2}} G^{p} \sum_{k = 1}^{\invtrunc_{n+1}} \tau(k)^{\frac{p}{2}-1} 
\left\lbrace
\left(\frac{\tau(k)}{\lfloor \tau(k) \rfloor}\right)^{\nicefrac{p}{2}}
+
\left(\frac{\tau(k)}{\lfloor \tau(k-1) \rfloor}\right)^{\nicefrac{p}{2}}
\right\rbrace\eqsp.
\end{aligned}
\end{equation}
The conclusion follows after plugging \eqref{eq:bound-telescopic-diff} in \eqref{eq:mom-2p-gen-bound}.
\end{proof}

The aforementioned results show that to control both the bias and higher-order moments, the distribution $\mu$ must be chosen so that it assigns positive mass to all positive integers in $\mathbb{N}$, and ensures that the ratio $\tau(k)/\tau(k-1)$, $k \geq 1$, remains uniformly bounded. 
More importantly, an appropriate trade-off is required between the rate at which the bias decreases and the asymptotic behavior of the summation terms appearing in Lemma~\ref{lemma:mom-2p-gen}. 
The geometric distribution stands as a natural choice, as it yields a bias decay of order $1/T_n$ while maintaining second-moment growth of order $\log T_n$.

\subsection{Proof of Proposition~\ref{prop1}}

\begin{proof}
For the Geometric distribution with parameter $q \in (0, 1)$, the function $\mu : k \mapsto (1 -q)^{k-1}q$ is non-increasing on $\mathbb{N}_{0}$. 
When $q = 1/2$, for any $k \geq 1$ and $n\geq 1$
\begin{equation*}
\tau(k) = 2^k, \qquad
\invtrunc_n = \left\lfloor \frac{\log \trunc_n}{\log 2} \right\rfloor\eqsp.
\end{equation*}

Assuming H\ref{hyp:stat}--H\ref{hyp:geomerg} hold, Proposition~\ref{prop1} then follows from Lemma \ref{lemma:bias-gen} and Lemma \ref{lemma:mom-2p-gen}.

\paragraph{Control of the bias.} Lemma \ref{lemma:bias-gen} yields that
there exist $\alpha \in [0,1)$ and $\zeta<\infty$ such that for all $n \geq 0$,
\begin{equation*}
\left\lVert \mathbb{E} \left[
     \Hmlmc \mid \mathcal{F}_{n}  
    \right] -   h\left(\theta_{n}\right) \right\rVert \leq  \frac{2\zeta G\sqrt{d}}{(1 - \alpha) 2^{\invtrunc_{n+1}}}\eqsp.
\end{equation*}
By definition of $\invtrunc_{n+1}$, we have $2^{1 + \invtrunc_{n+1}} > \trunc_{n+1}$ and then
\begin{equation*}
\left\lVert 
\mathbb{E} \left[\Hmlmc \mid \mathcal{F}_{n}\right] 
- h\left(\theta_{n}\right) 
\right\rVert 
\leq c_1 \frac{G}{\trunc_{n+1}}, \qquad c_1 = \frac{4\zeta \sqrt{d}}{(1 - \alpha)}\eqsp.
\end{equation*}

\paragraph{Control of higher order moments.} 
From Lemma \ref{lemma:mom-2p-gen}, there exist a positive constant $\beta$, $\alpha \in [0,1)$ and $\zeta<\infty$ such that for all $n \geq 0$, $p \geq 2$
\begin{equation*}
\mathbb{E}\left[
\left\lVert \Hmlmc \right\rVert^{p}
\right] 
\leq
2^{p-1} G^{p} \left[ 1 + 2^{2p}
+ \frac{16^{p}}{8}
\left\{
p\Gamma\left(\frac{p}{2}\right)\left(\frac{d}{\beta}\right)^{\nicefrac{p}{2}}(1 + 2^{\nicefrac{p}{2}})
+ \left(\frac{\zeta\sqrt{d}}{1 - \alpha}\right)^p(1 + 2^{p})
\right\}
\right]\sum_{k = 1}^{\invtrunc_{n+1}} 2^{\frac{kp}{2}-k}\eqsp.
\end{equation*}
where we used that
\begin{equation*}
\Delta(k, p) = 1 + 2^p,
\qquad
\sum_{k = 1}^{\invtrunc_n} \frac{1}{2^k} \leq 1,
\qquad
\sum_{k = 1}^{\invtrunc_n} 2^{\frac{kp}{2}-k} \geq 1\eqsp.
\end{equation*}
Then, for $p = 2$
\begin{equation*}
\mathbb{E}\left[ \left\lVert \Hmlmc \right\rVert^{2} \right] 
\leq 
\widetilde{c}_2 G^2 \invtrunc_{n+1},
\qquad
\widetilde{c}_2 =
34 + 64d\left(\frac{6}{\beta} + 5\frac{\zeta^{2}}{(1 - \alpha)^{2}}\right)\eqsp.
\end{equation*}
Since for the geometric distribution $\invtrunc_{n+1} = \lfloor \log\trunc_{n+1} / \log 2\rfloor$,
\begin{equation*}
\mathbb{E}\left[ \left\lVert \Hmlmc \right\rVert^{2} \right] 
\leq 
c_2 G^2 \log\trunc_{n+1},
\qquad
c_2 =
\frac{1}{\log 2}\left\lbrace 34 + 64d\left(\frac{6}{\beta} + 5\frac{\zeta^{2}}{(1 - \alpha)^{2}}\right)\right\rbrace\eqsp.
\end{equation*}
In addition, for $p = 3$, since by definition $\invtrunc_{n+1} \leq \log\trunc_{n+1} / \log 2$, we get
\begin{align*}
\mathbb{E}\left[ \left\lVert \Hmlmc \right\rVert^{3} \right] 
& \leq 
c_3 G^3 \sqrt{2}^{\invtrunc_{n+1}} 
\\
& \leq c_3 G^3 \sqrt{\trunc_{n+1}},
\qquad
c_3 =
\frac{4\sqrt{2}}{\sqrt{2} - 1}\left[
65 + 512d^{\nicefrac{3}{2}}\left(\frac{3\Gamma(1.5)(1+ 2^{\nicefrac{3}{2}})}{\beta^{\nicefrac{3}{2}}} + 9\frac{\zeta^{3}}{(1 - \alpha)^{3}}\right)
\right]\eqsp.
\end{align*}
\end{proof}

\section{Rates of convergence of the MLMC stochastic optimization}

In what follows, we focus on the Geometric distribution with parameter $1/2$ and denote,
\begin{equation}
\label{eq:def:kappa}
    \invtrunc_n = \max\{k \in \mathbb{N}\, :\, 2^k \leq T_{n}\} = \log_2(T_n)\eqsp. 
\end{equation}

\subsection{Proof of Theorem \ref{theo::gen}}
\label{sec:app:proof-theo-gen}

\begin{proof}
The result follows from Theorem 4.2 in \citet{surendran2024non}. 
Under Assumptions A\ref{A-smooth}--A\ref{A-bias-var}, it is sufficient to verify that Assumption \textbf{H3} of \citet{surendran2024non} holds.
However, note that Assumption \textbf{H3} in \citet{surendran2024non} requires their sequences $(\lambda_n)_{n\geq 1}$, $(r_n)_{n\geq 1}$ and $(\sigma^2_n)_{n\geq 1}$ to be non-increasing.
This non-increasing condition is not necessary to their proof, and the theorem remains valid without it. Accordingly, in what follows we establish the existence of these sequences without imposing the non-increasing assumption.

Let $n\geq 0$. Since $A_{n}$ is $\mathcal{F}_{n}$-measurable
\begin{align*}
& \mathbb{E} \left[ \left\langle \nabla V(\theta_{n}) , 
A_{n} \Hmlmc \right\rangle \right]
\\
& \qquad\qquad 
= \mathbb{E} \left[ \left\langle \nabla V(\theta_{n}) , A_{n}  \mathbb{E} \left[ \Hmlmc \mid \mathcal{F}_{n} \right]  \right\rangle \right] 
\\
& \qquad\qquad 
=  \mathbb{E} \left[ 
\left\langle \nabla V(\theta_{n}) , A_{n} \nabla V(\theta_{n})  \right\rangle 
\right]  
+  \mathbb{E} \left[ 
\left\langle \nabla V(\theta_{n}) ,  A_{n} \left\lbrace  \mathbb{E}\left[\Hmlmc\mid \mathcal{F}_{n}\right] -   \nabla V(\theta_{n}) \right\rbrace  \right\rangle
\right]\eqsp.
\end{align*}
The matrix $A_n$ is self-adjoint, then
\begin{equation*}
\mathbb{E} \left[ 
\left\langle \nabla V(\theta_{n}) , A_{n} \nabla V(\theta_{n})  \right\rangle 
\right] \geq \mathbb{E}\left[\lambda_{\min}(A_n) \left\lVert \nabla V(\theta_n) \right\rVert^2\right]\eqsp.
\end{equation*}
Since $\nabla V(\theta_n) = h(\theta_{n})$, Assumption A\ref{A-bounded} yields that $\lVert \nabla V(\theta_n) \rVert \leq G$. The Cauchy–Schwarz inequality then yields
\begin{multline*}
\mathbb{E} \left[ 
\left\langle \nabla V(\theta_{n}) ,  A_{n} \left\lbrace  \mathbb{E}\left[\Hmlmc\mid \mathcal{F}_{n}\right] -   \nabla V(\theta_{n}) \right\rbrace  \right\rangle
\right]
\\
\geq - G \mathbb{E}\left[
\lambda_{\max}(A_{n}) \left\lVert  \mathbb{E}\left[\Hmlmc\mid \mathcal{F}_{n}\right] - h(\theta_n) \right\rVert
\right]\eqsp.
\end{multline*}
After adding the last two inequalities, Assumptions  A\ref{A-bias-var} and A\ref{A-eigen} lead to 
\begin{equation}
\label{eq:h3-mlmc-gen}
\mathbb{E} \left[
\left\langle \nabla V(\theta_{n}) , A_{n} \Hmlmc \right\rangle \right] 
\geq \underline{\lambda}_{n+1} \left(\mathbb{E} \left[ \left\lVert \nabla V(\theta_{n}) \right\rVert^{2} \right]
- c_{1} \frac{\overline{\lambda}_{n+1}}{\underline{\lambda}_{n+1} T_{n+1}} G^2\right)\eqsp.
\end{equation}
In addition, Assumption A\ref{A-bias-var} states
\begin{equation}
\label{maj::mom}
\mathbb{E}\left[
    \left\lVert 
     \Hmlmc   \right\rVert^{2}
    \right]  \leq c_2G^{2}\log(T_{n+1})\eqsp.
\end{equation}
The result is then an application of Theorem 4.2 from \citet{surendran2024non} with, for $n = 0, \ldots, N$,
\begin{equation*}
w_{n+1} = 1,
\quad\delta_{n+1} = \frac{L}{2}\gamma_{n+1}^2 \overline{\lambda}_{n+1}^2,
\quad\lambda_{n+1} = \underline{\lambda}_{n+1},
\quad r_{n + 1} = c_{1} \frac{\overline{\lambda}_{n+1} }{\underline{\lambda}_{n+1} T_{n+1}}G^2,
\quad \sigma_{n}^2 = c_2G^{2}\log(T_{n+1})\eqsp.
\end{equation*}
\end{proof}

\subsection{Proof of Corollary \ref{cor:gen}}

\begin{proof}
The result follows from Theorem \ref{theo::gen} and using that there exist positive constants $M_1$ and $M_2$ such that
\begin{equation*}
\sum_{k = 1}^{N + 1} \frac{1}{k^\eta} \leq M_1
\begin{cases}
N^{1 - \eta} & \text{, if} \quad\eta < 1,
\\
\log N & \text{, if} \quad\eta = 1,
\\
1 & \text{, if} \quad\eta > 1,
\end{cases}
\qquad
\sum_{k = 1}^{N + 1} \frac{1}{k^\eta}\log k \leq M_2
\begin{cases}
N^{1 - \eta}\log N & \text{, if} \quad\eta < 1,
\\
(\log N)^2 & \text{, if} \quad\eta = 1,
\\
1 & \text{, if} \quad\eta > 1.
\end{cases}
\end{equation*}
\end{proof}

\subsection{Proof of Theorem \ref{theo::adagrad}}
\label{sec:app:proof-adagrad}

\begin{proof}
Let $n\geq 0$. As in the proof of Theorem \ref{theo::gen} (see Appendix \ref{sec:app:proof-theo-gen}), our aim is to prove that Assumption \textbf{H3} in \citet{surendran2024non} hold without imposing the non-increasing condition.
For all $n \geq 0$, following the proof of Corollary 4.5 in \citet{surendran2024non}, we have for the conditioning matrix $A_n$ as defined in \eqref{eq:a-adagrad},
\begin{multline*}
\mathbb{E} \left[ 
\left\langle \nabla V(\theta_{n}) , A_{n} \Hmlmc  \right\rangle \mid \mathcal{F}_{n} 
\right]   
=
\left\langle \nabla V(\theta_{n}) , \widetilde{A}_{n}  \mathbb{E} \left[ \Hmlmc \mid \mathcal{F}_{n} \right]  \right\rangle
\\ + \mathbb{E} \left[ 
\left\langle \nabla V(\theta_{n}) , (A_{n} - \widetilde{A}_{n}) \Hmlmc  \right\rangle \mid \mathcal{F}_{n} 
\right]\eqsp,
\end{multline*}
where $\widetilde{A}_{n}$ is a $\mathcal{F}_n$-measurable matrix defined as
\begin{equation*}
\widetilde{A}_0 = \varepsilon_1 I_{d},
\quad
\widetilde{A}_{n} = \left\lbrace \frac{1}{\varepsilon_{n+1}^2} I_{d} + \text{Diag} \left( \frac{1}{n+1}\sum_{k=0}^{n-1} \widetilde H^{(k)}
\right)   \right\rbrace^{-1/2}\eqsp.
\end{equation*}
We have 
$\lambda_{\min}(\widetilde{A}_0) = \lambda_{\max}(\widetilde{A}_0) = \varepsilon_1$ and for all $n\geq 1$.
\begin{equation*}
\lambda_{\min}(\widetilde{A}_n) \geq
\underline{\varepsilon}_n := \frac{1}{\sqrt{\varepsilon_{n+1}^{-2} + \displaystyle\sup_{0 \leq k \leq n-1} M_k^2}},
\qquad
\lambda_{\max}(\widetilde{A}_n) \leq \varepsilon_{n+1}\eqsp.
\end{equation*}
$\widetilde{A}_n$ thus fulfills Assumption A\ref{A-eigen}.
Since Assumption A\ref{A-smooth} to A\ref{A-bias-var} also hold, by applying Equation \eqref{eq:h3-mlmc-gen} from the proof of Theorem \ref{theo::gen} to $\widetilde{A}_n$, we get, writing $\underline{\varepsilon}_0 = \varepsilon_1$,
\begin{equation*}
\mathbb{E} \left[
\left\langle \nabla V(\theta_{n}) , \widetilde{A}_{n} \Hmlmc \right\rangle \right] 
\geq 
\underline{\varepsilon}_n
\left(\mathbb{E} \left[ \left\lVert \nabla V(\theta_{n}) \right\rVert^{2} \right]
- c_{1} \frac{\varepsilon_{n+1}}{\underline{\varepsilon}_{n} T_{n+1}} G^2\right)\eqsp.
\end{equation*}
Moreover, following the proof of Corollary 4.5 in \citet{surendran2024non}, we have
\begin{equation*}
\mathbb{E} \left[ 
\left\langle \nabla V(\theta_{n}) , (A_{n} - \widetilde{A}_{n}) \Hmlmc  \right\rangle \mid \mathcal{F}_{n} 
\right]
\geq -\frac{\varepsilon_{n+1}^3}{n + 1} \left\lVert \nabla V(\theta_{n}) \right\rVert \mathbb{E}\left[\left\lVert \Hmlmc \right\rVert^3 \mid \mathcal{F}_n\right]\eqsp.
\end{equation*}
Since $\nabla V(\theta_n) = h(\theta_{n})$, Assumption A\ref{A-bounded} yields that $\lVert \nabla V(\theta_n) \rVert \leq G$.
Then, using Assumption A\ref{A-moment-3}, and integrating both sides yields
\begin{align*}
\mathbb{E} \left[ 
\left\langle \nabla V(\theta_{n}) , (A_{n} - \widetilde{A}_{n}) \Hmlmc  \right\rangle 
\right]
& \geq  -\frac{\varepsilon_{n+1}^3}{n + 1} c_3 G^4 \sqrt{T_{n+1}}\eqsp.
\end{align*}
In conclusion,
\begin{equation*}
\mathbb{E} \left[ 
\left\langle \nabla V(\theta_{n}) , A_{n} \Hmlmc  \right\rangle
\right]   
\geq 
\underline{\varepsilon}_n
\left(\mathbb{E} \left[ \left\lVert \nabla V(\theta_{n}) \right\rVert^{2} \right]
- c_{1} \frac{\varepsilon_{n+1}}{\underline{\varepsilon}_{n} T_{n+1}} G^2
 -\frac{\varepsilon_{n+1}^3 c_3}{(n + 1)\underline{\varepsilon}_n}  G^4 \sqrt{T_{n+1}}
\right)\eqsp,
\end{equation*}
which, along with the second order moment from Assumption \ref{A-bias-var}, proves that assumption \textbf{H3} in \citet{surendran2024non} hold.
The result is then an application of Theorem 4.2 from \citet{surendran2024non} with, for $n = 0, \ldots, N$, $w_{n+1} = 1$
\begin{equation*}
\delta_{n+1} = \frac{L}{2}\gamma_{n+1}^2 \varepsilon_{n+1}^2,
\quad\lambda_{n+1} = \underline{\varepsilon}_{n},
\quad r_{n + 1} = c_{1} \frac{\varepsilon_{n+1}}{\underline{\varepsilon}_{n} T_{n+1}} G^2
+ \frac{\varepsilon_{n+1}^3 c_3}{(n + 1)\underline{\varepsilon}_n}  G^4 \sqrt{T_{n+1}},
\quad \sigma_{n}^2 = c_2G^{2}\log(T_{n+1})\eqsp.
\end{equation*}
\end{proof}


\subsection{Proof of Theorem \ref{theo::amsgrad}}
\label{sec:app-proof-amsgrad}
\begin{proof}
While the proof largely follows the same lines as the proof of Theorem 4.6 in \citet{surendran2024non}, it requires additional technical steps since we do not impose its Assumption \textbf{H3}(i). In particular, we must directly address the fact that $A_n$ is not $\mathcal{F}_n$ measurable but $\mathcal{F}_{n+1}$-measurable.

Let $\rho_1 \in(0, 1)$ and denote $\kappa = \rho_1/(1-\rho_1)$. Given the sequence $(\theta_n)_{n \geq 0}$ from Algorithm \ref{algo::MLMC-AMSGrad}, we define $\widetilde{\theta}_0 = \theta_0$, and for $n \geq 0$
\begin{equation*}
\widetilde{\theta}_{n+1} = \theta_{n+1} + \kappa \left(\theta_{n+1} - \theta_{n}\right)\eqsp.
\end{equation*}
Then, using the definitions of $(\theta_n)_{n \geq 0}$, $(m_n)_{n \geq -1}$, and $(A_n)_{n \geq -1}$, we get for any $n \geq 0$,
\begin{align}
\widetilde{\theta}_{n+1} - \widetilde{\theta}_{n}
&  =  (1+\kappa) \left( \theta_{n+1} - \theta_{n} \right) - \kappa \left( \theta_{n} - \theta_{n-1} \right) \nonumber
\\
& = - (1+ \kappa ) \gamma_{n+1}A_{n}m_{n} + \kappa \gamma_{n}A_{n-1}m_{n-1} \nonumber
\\
& = \kappa \left( \gamma_{n} A_{n-1} - \gamma_{n+1}A_{n} \right)m_{n-1} - \gamma_{n+1}A_{n} \Hmlmc\eqsp. \label{eq:amsgrad-diff-theta}
\end{align}
Since $V$ is $L$-smooth, it satisfied for all $\theta, \theta'\in\mathbb{R}^d$,
\begin{equation*}
V(\theta) \leq V(\theta')+ \left\langle \nabla V(\theta'), \theta - \theta' \right\rangle + \frac{L}{2} \left\lVert \theta - \theta' \right\rVert^{2}\eqsp.
\end{equation*}
Then, using \eqref{eq:amsgrad-diff-theta} and the identity $\lVert a + b \rVert^2 \leq 2 \lVert a \rVert^2 + 2 \lVert b \rVert^2$, we get
\begin{align}
V \left( \widetilde{\theta}_{n+1} \right)
& \leq V \left( \widetilde{\theta}_{n} \right) - \gamma_{n+1} \left\langle \nabla V \left( \widetilde{\theta}_{n} \right) , A_{n} \Hmlmc \right\rangle  + \kappa \left\langle \nabla V \left( \widetilde{\theta}_{n} \right) , \left( \gamma_{n}A_{n-1} - \gamma_{n+1}A_{n} \right)m_{n-1} \right\rangle \nonumber
\\
& \qquad\qquad + L\gamma_{n+1}^{2} \left\lVert A_{n} \Hmlmc \right\rVert^{2} + L \kappa^{2} \left\lVert \left( \gamma_{n}A_{n-1} - \gamma_{n+1}A_{n} \right)m_{n-1} \right\rVert^{2} \\
& \leq V \left( \widetilde{\theta}_{n} \right) + S_{1,n} + S_{2,n} + S_{3,n} + S_{4,n} + S_{5, n}\eqsp, \label{eq:bound-vtilde}
\end{align}
with
\begin{align*}
    S_{1,n} & = - \gamma_{n+1} \left\langle \nabla V(\theta_{n}) , A_{n-1}\Hmlmc \right\rangle + L \gamma_{n+1}^{2} \left\lVert A_{n} \Hmlmc  \right\rVert^{2}\eqsp, \\
    S_{2,n} & = - \gamma_{n+1} \left\langle \nabla V \left( \widetilde{\theta}_{n} \right) {- \nabla V \left( \theta_{n} \right)}, A_{n} \Hmlmc \right\rangle\eqsp, \\
    S_{3,n} & = \kappa  \left\langle \nabla V \left( \widetilde{\theta}_{n} \right) , (\gamma_{n}A_{n-1} - \gamma_{n+1}A_{n})m_{n-1} \right\rangle\eqsp, \\
    S_{4,n} & = L \kappa^{2} \left\lVert ( \gamma_{n}A_{n-1} - \gamma_{n+1}A_{n} )m_{n-1} \right\rVert^{2}\eqsp,
    \\
    S_{5,n} & = - \gamma_{n+1} \left\langle \nabla V(\theta_{n}) , (A_n - A_{n-1}) \Hmlmc\right\rangle\eqsp.
\end{align*}
We thus have, using that $\widetilde{\theta}_0 = \theta_0$,
\begin{equation}
\label{eq:start-ams}
\sum_{n = 0}^N \mathbb{E}\left[
V \left( \widetilde{\theta}_{n+1} \right)
-  V \left( \widetilde{\theta}_{n} \right)
\right] = 
\mathbb{E}\left[V \left( \widetilde{\theta}_{N+1} \right)
- V(\theta_{0})\right]
\leq \sum_{n = 0}^N \mathbb{E}\left[S_{1,n} + S_{2,n} + S_{3,n} + S_{4,n} + S_{5, n}\right]\eqsp.
\end{equation}

\paragraph{Preliminary result.} Let $n \geq 1$. The sequences $(A_{n})_{n \geq 0}$, $(W_{n})_{n \geq 0}$ and $(\widehat{W}_{n})_{n \geq 0}$ as defined in Algorithm \ref{algo::MLMC-AMSGrad} consist of diagonal matrices. 
Consequently, from their definitions, we have
\begin{align*}
\lambda_{\min}(A_{n-1}) & = \left\lbrace 
\delta + \lambda_{\max}\left(\widehat{W}_{n-1}\right) \right\rbrace^{-\nicefrac{1}{2}}
\geq \left[ 
\delta + \max\left\lbrace
\lambda_{\max}\left(W_{0}\right), \ldots, \lambda_{\max}\left(W_{n-1}\right)
\right\rbrace\right]^{-\nicefrac{1}{2}}\eqsp,
\\
\lambda_{\max}(A_{n-1}) & = \left\lbrace
\delta + \lambda_{\min}\left(\widehat{W}_{n-1}\right) \right\rbrace^{-\nicefrac{1}{2}} 
\leq \left\lbrace
\delta + \lambda_{\min}\left(W_{n-1}\right) 
\right\rbrace^{-\nicefrac{1}{2}}\eqsp.
\end{align*}
Note that,
\begin{equation*}
W_{n-1} = (1-\rho_2)\sum_{k = 0}^{n-1}\rho_2^{n-1-k} w_k,
\quad w_k = \min\left\{\varepsilon_{k+1} I_d, \text{Diag}\left(\Hmlmc[k][k+1]\Hmlmc[k][k+1]^{\top}\right)\right\}\eqsp.
\end{equation*}
Then, since the sequence $(\varepsilon_n)_{n\geq 1}$ is positive, $\lambda_{\min}(W_{n-1}) \geq 0$ and
\begin{equation}
\label{eq:lmax-amsgrad}
\lambda_{\max}(A_{n-1}) \leq \frac{1}{\sqrt{\delta}}\eqsp.
\end{equation}
Moreover, since the sequence $(\varepsilon_n)_{n\geq 1}$ is non-decreasing and $\rho_2\in (0,1)$
\begin{equation*}
\lambda_{\max}(W_{n-1}) \leq (1-\rho_2)\sum_{k = 0}^{n-1}\rho_2^{n-1-k} \varepsilon_{k+1}
\leq (1-\rho_2)\varepsilon_n \sum_{k = 0}^{n-1}\rho_2^{n-1-k}
\leq \varepsilon_n(1 - \rho_2^n)\eqsp,
\end{equation*}
and
\begin{equation}
\label{eq:lmin-amsgrad}
\lambda_{\min}(A_{n-1}) \geq \frac{1}{\sqrt{\delta + \varepsilon_n(1 - \rho_2^n)}} := \underline{\varepsilon}_n\eqsp.
\end{equation}

\paragraph{Bounding $S_{1,n}$.} 
Using Assumption \ref{A-bias-var} and \eqref{eq:lmax-amsgrad}, we have for all $n \geq 0$
\begin{equation*}
L \gamma_{n+1}^{2} \mathbb{E}\left[\left\lVert A_{n} \Hmlmc \right\rVert^{2}\right] \leq \frac{L}{\delta} \gamma_{n+1}^{2} c_2G^2\log T_{n+1}\eqsp.
\end{equation*}

In addition, for all $n \geq 1$, since $A_{n-1}$ is $\mathcal{F}_{n}$ measurable, following the same arguments as in the proof of Theorem \ref{theo::gen} (see Appendix \ref{sec:app:proof-theo-gen}), we have
\begin{align*}
\mathbb{E} \left[ \left\langle \nabla V(\theta_{n}) , 
A_{n-1} \Hmlmc \right\rangle \right]
& \geq \mathbb{E}\left[\lambda_{\min}(A_{n-1})\left\lVert \nabla V(\theta_n)\right\rVert^2\right] - c_1 \frac{G^2}{T_{n+1}} \mathbb{E}\left[\lambda_{\max}(A_{n-1})\right]\eqsp.
\end{align*}

Consequently,
\begin{align*}
\mathbb{E} \left[ \left\langle \nabla V(\theta_{n}) , 
A_{n-1} \Hmlmc \right\rangle \right]
& \geq \underline{\varepsilon}_n \mathbb{E}\left[\left\lVert \nabla V(\theta_n)\right\rVert^2\right] - c_1 \frac{G^2}{T_{n+1}\sqrt{\delta}}\eqsp.
\end{align*}

In conclusion, since $A_{-1} = 0$,
\begin{equation*}
S_{1, 0} = L \gamma_{1}^{2} \left\lVert A_{0} \Hmlmc[0][1]  \right\rVert^{2}
\leq \frac{L}{\delta} \gamma_{1}^{2} c_2G^2\log T_{1}\eqsp,
\end{equation*}
and for all $n\geq 1$
\begin{equation*}
\mathbb{E}\left[S_{1,n}\right] \leq 
c_1 \gamma_{n+1}\frac{G^2}{T_{n+1}\sqrt{\delta}}
- \underline{\varepsilon}_n \gamma_{n+1} \mathbb{E}\left[\left\lVert \nabla V(\theta_n)\right\rVert^2\right]
+ \frac{L}{\delta} \gamma_{n+1}^{2} c_2G^2\log T_{n+1}\eqsp.
\end{equation*}
\paragraph{Bounding $S_{2,n}$.} Let $n \geq 1$. The Cauchy-Schwarz inequality and the identity: $xy \leq x^2/2 + y^2/2$, for all $x, y\in\mathbb{R}$, yields
\begin{align*}
\mathbb{E}\left[S_{2,n}\right] & = \mathbb{E}\left[
\left\langle \nabla V \left( \widetilde{\theta}_{n} \right) - \nabla V \left( \theta_{n} \right), -\gamma_{n+1} A_{n} \Hmlmc \right\rangle
\right]
\\
& \leq \frac{1}{2}\mathbb{E}\left[\left\lVert
\nabla V \left( \widetilde{\theta}_{n} \right) - \nabla V \left( \theta_{n} \right)
\right\rVert^2\right]
+ \frac{\gamma_{n+1}^2}{2} \mathbb{E}\left[\left\lVert
 A_{n} \Hmlmc
\right\rVert^2\right]\eqsp.
\end{align*}
Using Assumptions \ref{A-smooth} and \ref{A-bias-var}, and \eqref{eq:lmax-amsgrad}, we get
\begin{align*}
\mathbb{E}\left[S_{2,n}\right]
& \leq \frac{L^2}{2}\mathbb{E}\left[\left\lVert
\widetilde{\theta}_{n} - \theta_{n}
\right\rVert^2\right]
+ c_2\gamma_{n+1}^2\frac{G^2}{2\delta} \log T_{n+1}
\\
& \leq \frac{L^2\kappa^2}{2}\mathbb{E}\left[\left\lVert
\theta_{n} - \theta_{n-1}
\right\rVert^2\right]
+ c_2\gamma_{n+1}^2\frac{G^2}{2\delta} \log T_{n+1}
\\
& \leq \gamma_n^2\frac{L^2\kappa^2}{2} \mathbb{E}\left[\left\lVert
A_{n-1}m_{n-1}
\right\rVert^2\right]
+ c_2\gamma_{n+1}^2\frac{G^2}{2\delta} \log T_{n+1}\eqsp.
\end{align*}
Since Assumption A\ref{A-bias-var} holds and $(T_n)_{n\geq 1}$ is non-decreasing, after using \eqref{eq:lmax-amsgrad}, it follows from Lemma \ref{lemma:mn-amsgrad}
\begin{equation*}
\mathbb{E}\left[\left\lVert
A_{n-1}m_{n-1}
\right\rVert^2\right]
\leq c_2 \frac{(1 - \rho_1)(1 - \rho_1^n)}{\delta} G^2 \log T_n\eqsp.
\end{equation*}
Using the definition of $\kappa$, it follows that for all $n\geq 1$
\begin{equation*}
\mathbb{E}\left[S_{2,n}\right]
\leq 
c_2\frac{G^2}{2\delta}\left\lbrace
L^2\frac{\rho_1^2 (1 - \rho_1^n)}{1 - \rho_1} \gamma_n^2 \log T_{n}
+ \gamma_{n+1}^2  \log T_{n+1}
\right\rbrace.
\end{equation*}
In addition, since $\widetilde{\theta}_0 = \theta_0$, we have $S_{2, 0}$ = 0. 

\paragraph{Bounding $S_{3,n}$.}
Let $n\geq 1$. Since $\nabla V(\theta_n) = h(\theta_{n})$, Assumption A\ref{A-bounded} yields that $\lVert \nabla V(\theta_n) \rVert \leq G$.
Applying the Cauchy–Schwarz inequality then yields
\begin{equation*}
\mathbb{E}\left[S_{3, n}\right] \leq 
\kappa G \mathbb{E}\left[
\left\lVert (\gamma_{n}A_{n-1} - \gamma_{n+1}A_{n})m_{n-1} \right\lVert
\right]\eqsp.
\end{equation*}
Since the matrices $(A_{n})_{n\geq 0}$ are diagonal, and since for all $a,b,c > 0$, $2ab \leq a^{2}/c + cb^{2}$, we have for any $a_n > 0$
\begin{align*}
\left\lVert (\gamma_{n}A_{n-1} - \gamma_{n+1}A_{n})m_{n-1} \right\rVert 
& \leq  \sum_{i=1}^{d}\left\lVert m_{n-1} \right\rVert \left\lvert \gamma_{n}e_i^\top A_{n-1} e_i - \gamma_{n+1} e_i^\top A_{n} e_i \right\rvert  \\
& \leq \frac{1}{2a_{n}} \left\lVert m_{n-1} \right\rVert^{2} + \frac{a_{n}}{2} \sum_{i=1}^{d} \left( \gamma_{n}e_i^\top A_{n-1} e_i - \gamma_{n+1} e_i^\top A_{n} e_i \right)^{2}\eqsp.
\end{align*}
The sequence $(\gamma_n)_{n\geq 1}$ is non-increasing, and by construction of the matrices $\widehat{W}_{n}$ in Algorithm \ref{algo::MLMC-AMSGrad}, for any $1 \leq i \leq d$, $e_i^\top A_{n} e_i \leq e_i^\top A_{n-1} e_i$. Thus,
\begin{equation*}
\gamma_{n}e_i^\top A_{n-1} e_i \geq \gamma_{n+1} e_i^\top A_{n} e_i\eqsp.
\end{equation*}
Using that if $x \geq y$, $(x-y)^{2} \leq x^{2} - y^{2}$, we have
\begin{align*}
\left\lVert (\gamma_{n}A_{n-1} - \gamma_{n+1}A_{n})m_{n-1} \right\rVert
& \leq \frac{1}{2a_{n}} \left\lVert m_{n-1} \right\rVert^{2} + \frac{a_{n}}{2} \sum_{i=1}^{d} \left( \gamma_{n}e_i^\top A_{n-1} e_i \right)^2 - \left(\gamma_{n+1} e_i^\top A_{n} e_i \right)^{2}\eqsp.
\end{align*}
Using Lemma \ref{lemma:mn-amsgrad}, it directly follows
\begin{align*}
\sum_{n = 1}^N \mathbb{E}\left[S_{3,n}\right] 
& \leq
\kappa G \left(\sum_{n = 1}^N \frac{1}{2a_{n}} \mathbb{E}\left[\left\lVert m_{n-1} \right\rVert^{2}\right] + 
\sum_{i=1}^{d} \sum_{n = 1}^N \frac{a_{n}}{2} \mathbb{E}\left[\left( \gamma_{n}e_i^\top A_{n-1} e_i \right)^2 - \left(\gamma_{n+1} e_i^\top A_{n} e_i \right)^{2}\right]\right)
\\
& \leq 
c_2\frac{\kappa(1 - \rho_1)G^3}{2}\sum_{n = 1}^N \frac{1 - \rho_1^n}{a_{n}} \log T_n  + 
\frac{\kappa G}{2}\sum_{i=1}^{d} \mathbb{E}\left[\sum_{n = 1}^N a_{n} \left\lbrace \left( \gamma_{n}e_i^\top A_{n-1} e_i \right)^2 - \left(\gamma_{n+1} e_i^\top A_{n} e_i \right)^{2}\right\rbrace \right]\eqsp.
\end{align*}
We set $a_n = 1/\gamma_{n}^2$. Since the sequence $(\gamma_n)_{n\geq 1}$ is non-increasing, the second term of the right hand side can be written as
\begin{align*}
& \frac{\kappa G}{2}\sum_{i=1}^{d} \mathbb{E}\left[
\sum_{n = 1}^N 
\frac{1}{\gamma_n^2} \left\lbrace 
\left( \gamma_{n}e_i^\top A_{n-1} e_i \right)^2 
- \left(\gamma_{n+1} e_i^\top A_{n} e_i \right)^{2} 
\right\rbrace 
\right]
\\
& \qquad\qquad = \frac{\kappa G}{2}\sum_{i=1}^{d} \mathbb{E}\left[
\sum_{n = 1}^N \left\lbrace \left(e_i^\top A_{n-1} e_i \right)^2 - \left(e_i^\top A_{n} e_i \right)^{2} \right\rbrace
+ \left(1 - \frac{\gamma_{n+1}^2}{\gamma_n^2}\right) \left(e_i^\top A_{n} e_i \right)^{2}
\right]
\\
& \qquad\qquad \leq \frac{d \kappa G}{2}  \mathbb{E}\left[\lambda_{\max}(A_{0})^2\right]
+ \frac{d \kappa G}{2}
\sum_{n = 1}^N  \left(1 - \frac{\gamma_{n+1}^2}{\gamma_n^2}\right) \mathbb{E}\left[\lambda_{\max}(A_{n})^2\right]\eqsp.
\end{align*}
Using for any $n\geq 1$ the upper-bound \eqref{eq:lmax-amsgrad} on the maximal eigenvalue of $A_n$, we get
\begin{equation*}
\sum_{n = 1}^N \mathbb{E}\left[S_{3,n}\right] 
\leq c_2\frac{\kappa(1 - \rho_1)G^3}{2}\sum_{n = 1}^N (1 - \rho_1^n)\gamma_n^2 \log T_n  
+\frac{d \kappa G}{2\delta} \left\lbrace 1 
+ 
\sum_{n = 1}^N  \left(1 - \frac{\gamma_{n+1}^2}{\gamma_n^2}\right)\right\rbrace\eqsp.
\end{equation*}
In addition, since $m_{-1} = 0$, we have $S_{3, 0} = 0$.

\paragraph{Bounding $S_{4,n}$.}
Let $n\geq 1$. The sequence $(\gamma_n)_{n\geq 1}$ is non-increasing, and by construction of the matrices $\widehat{W}_{n-1}$ in Algorithm \ref{algo::MLMC-AMSGrad}, for any $1 \leq i \leq d$, $e_i^\top A_{n} e_i \leq e_i^\top A_{n-1} e_i$. Thus,
\begin{equation*}
0 \leq \lambda_{\max}(\gamma_{n} A_{n-1} - \gamma_{n+1} A_{n}) 
\leq \lambda_{\max}(\gamma_{n} A_{n-1}) \leq \frac{\gamma_n}{\sqrt{\delta}}\eqsp.
\end{equation*}
From Lemma \ref{lemma:mn-amsgrad}, we have
\begin{align*}
\sum_{n = 1}^N \mathbb{E}\left[S_{4,n}\right] 
& \leq L \kappa^{2} \sum_{n = 1}^N \mathbb{E}\left[
\lambda_{\max}( \gamma_{n}A_{n-1} - \gamma_{n+1}A_{n} )^2\left\lVert m_{n-1} \right\rVert^{2}
\right]
\\
& \leq c_2 \kappa^{2} (1 - \rho_1) \frac{LG^2}{\delta}
\sum_{n = 1}^N
\gamma_{n}^2
(1 - \rho_1^n) \log T_n\eqsp.
\end{align*}
In addition, since $m_{-1} = 0$, we have $S_{4, 0} = 0$.

\paragraph{Bounding $S_{5,n}$.}
Let $n \geq 1$. Using that $\lVert \nabla V(\theta_n) \rVert \leq G$ (consequence of Assumption A\ref{A-bounded} and $\nabla V(\theta_n) = h(\theta_{n})$) together with the Cauchy-Schwarz inequality, yields
\begin{equation*}
\sum_{n = 1}^N\mathbb{E}\left[S_{5,n}\right]  \leq  G \sum_{n = 1}^N\gamma_{n+1} \mathbb{E}\left[\left\lVert (A_n - A_{n-1}) \Hmlmc\right\rVert\right]\eqsp.
\end{equation*}
Using the same arguments as for the term $S_{3,n}$, we get for any $b_n > 0$
\begin{equation*}
\left\lVert (A_n - A_{n-1}) \Hmlmc\right\rVert
\leq \frac{1}{2b_n}\left\lVert \Hmlmc\right\rVert^2
+ \frac{b_n}{2}\sum_{i=1}^{d} \left( e_i^\top A_{n-1} e_i\right)^2  - \left( e_i^\top A_{n} e_i \right)^{2}\eqsp.
\end{equation*}
Then, it follows from A\ref{A-moment-3}, that
\begin{equation*}
\sum_{n = 1}^N \mathbb{E}\left[S_{5,n}\right]  \leq  c_2 G^3 \sum_{n = 1}^N\frac{\gamma_{n+1}}{b_n}
 \log T_{n+1}
 + \frac{G}{2} \sum_{i=1}^{d} \mathbb{E}\left[\sum_{n = 1}^N \gamma_{n+1} b_n \left\lbrace \left( e_i^\top A_{n-1} e_i\right)^2  - \left( e_i^\top A_{n} e_i \right)^{2} \right\rbrace\right]\eqsp.
\end{equation*}
We set $b_n = 1/\gamma_{n+1}$. After simplifying the telescopic terms, using \eqref{eq:lmax-amsgrad}, we get
\begin{equation*}
\sum_{n = 1}^N \mathbb{E}\left[S_{5,n}\right]  \leq  c_2 G^3 \sum_{n = 1}^N \gamma_{n+1}^2
 \log T_{n+1}
 + \frac{dG}{2\delta}\eqsp.
\end{equation*}
In addition, since $A_{-1} = 0$, we get with the Cauchy-Schwarz inequality
\begin{align*}
\mathbb{E}\left[S_{5,0}\right] 
& \leq \gamma_1 \mathbb{E}\left[
\left\lVert \nabla V(\theta_0) \right\rVert \left\lVert A_0 \Hmlmc[0][1] \right\rVert
\right]
\\
& \leq \frac{\gamma_1 G}{\sqrt{\delta}}
\mathbb{E}\left[
\left\lVert \Hmlmc[0][1] \right\rVert
\right]\eqsp.
\end{align*}
Since $K_1$ and $\bX^{(1)}$ are independent, Assumption A\ref{A-bounded} yields
\begin{align*}
\mathbb{E}\left[
\left\lVert \Hmlmc[0][1] \right\rVert
\right] 
 = \mathbb{E}\left[
\sum_{k = 1}^{\infty}
\mathbb{P}[K_1 = k]
\mathbb{E}\left[
\left\lVert \Hhat[\theta_0][k][\bX^{(1)}] \right\rVert \mid \mathcal{F}_0
\right] 
\right]
& \leq
\sqrt{d} G \sum_{k = 1}^{\infty}
\frac{1}{2^k}
\left(1 + 2^{k+1} \mathds{1}_{2^k \leq T_1}\right)
\\
& \leq 
\sqrt{d} G \left(1 + 2\frac{\log T_1}{\log 2}\right)\eqsp.
\end{align*}

\paragraph{Conclusion.} From \eqref{eq:start-ams}, after combining the upper bounds related to $S_{1,n}$, $S_{2,n}$, $S_{3,n}$, $S_{4,n}$ and $S_{5,n}$, we get
\begin{align*}
\sum_{n = 1}^N \underline{\varepsilon}_n \gamma_{n+1} \mathbb{E}\left[\left\lVert \nabla V(\theta_n)\right\rVert^2\right]
& 
\leq \mathbb{E}\left[V(\theta_{0}) - V \left( \widetilde{\theta}_{N+1} \right) \right] 
+ \widetilde{b}_0 
+ b_1 \sum_{n = 1}^N\frac{\gamma_{n+1}}{T_{n+1}}
+ b_2 \sum_{n = 1}^{N}\gamma_{n+1}^{2}\log T_{n+1}
\\
& \qquad \qquad + b_3
\sum_{n = 1}^N
 (1 - \rho_1^n) \gamma_n^2 \log T_{n}
+ b_4 \sum_{n = 1}^N  \left(1 - \frac{\gamma_{n+1}^2}{\gamma_n^2}\right)\eqsp.
\end{align*}
where
\begin{align*}
\widetilde{b}_0 & = \frac{dG}{2\delta(1-\rho_1)}
+ \gamma_1 \sqrt{\frac{d}{\delta}} \left(1 + 2\frac{\log T_1}{\log 2}\right)G^2 + c_2\frac{L}{\delta}G^2 \gamma_{1}^{2}\log T_{1} \eqsp,
\\
b_1 & = c_1 \frac{G^2}{\sqrt{\delta}}\eqsp,
\\
b_2 & = \frac{c_2}{2\delta}\left(1 + 2L + 2\delta G\right)G^2\eqsp,
\\
b_3 & = \frac{c_2\rho_1^2}{2\delta(1-\rho_1)}\left( L^2
 + 2L + \frac{\delta(1-\rho_1)}{\rho_1} G\right)G^2\eqsp,
\\
b_4 & = \frac{d \rho_1 G}{2\delta(1-\rho_1)}\eqsp.
\end{align*}
For all $\underline{\varepsilon}_0 > 0$, since Assumption A\ref{A-bounded} holds, we have
\begin{equation*}
\underline{\varepsilon}_0 \gamma_{1} \mathbb{E}\left[\left\lVert \nabla V(\theta_0)\right\rVert^2\right] \leq \underline{\varepsilon}_0 \gamma_{1} G^2\eqsp.
\end{equation*}
Setting $b_0 = \widetilde{b}_0 + \underline{\varepsilon}_0 \gamma_{1} G^2$, and using that $V \left( \widetilde{\theta}_{N+1} \right) \geq V(\theta^{*})$, we get
\begin{align*}
\sum_{n = 0}^N \underline{\varepsilon}_n \gamma_{n+1} \mathbb{E}\left[\left\lVert \nabla V(\theta_n)\right\rVert^2\right]
& 
\leq \mathbb{E}\left[V(\theta_{0})\right] - V \left(\theta^{*}\right)
+ b_0 
+ b_1 \sum_{n = 1}^N\frac{\gamma_{n+1}}{T_{n+1}}
+ b_2 \sum_{n = 1}^{N}\gamma_{n+1}^{2}\log T_{n+1}
\\
& \qquad \qquad + b_3
\sum_{n = 1}^N
 (1 - \rho_1^n) \gamma_n^2 \log T_{n}
+ b_4 \sum_{n = 1}^N  \left(1 - \frac{\gamma_{n+1}^2}{\gamma_n^2}\right)\eqsp.
\end{align*}
The conclusion directly follows.
\end{proof}

\begin{lemma}
\label{lemma:mn-amsgrad}
Assume A\ref{A-bias-var} holds. For all  positive non-decreasing sequence $(T_n)_{n\geq 1}$, the sequence $(m_n)_{n \geq 0}$ as defined in Algorithm \ref{algo::MLMC-AMSGrad} satisfies for all $n\geq 1$
\begin{equation*}
\mathbb{E}\left[\left\lVert
m_{n-1}
\right\rVert^2\right]
\leq c_2(1 - \rho_1)(1 - \rho_1^n)G^2 \log T_n\eqsp.
\end{equation*}
\end{lemma}

\begin{proof}
By definition, we have
\begin{equation}
\label{eq:def-mn}
m_{n-1} = (1-\rho_1)\sum_{k = 0}^{n-1} \rho_1^{n-1-k} \Hmlmc[k][k+1]\eqsp.
\end{equation}
Then, using Minkowski inequality, and noting that $\rho_1>0$, we get
\begin{equation*}
\mathbb{E}\left[\left\lVert
m_{n-1}
\right\rVert^2\right]
\leq (1 - \rho_1)^2
\left\lbrace
\sum_{k = 0}^{n-1} \rho_1^{n-1-k} \mathbb{E}\left[\left\lVert
  \Hmlmc[k][k+1]
\right\rVert^2\right]^{\nicefrac{1}{2}}
\right\rbrace^2\eqsp.
\end{equation*}
Since the sequence $(T_n)_{n\geq 1}$ is non-decreasing, it follows from Assumption \ref{A-bias-var}
\begin{equation*}
\mathbb{E}\left[\left\lVert
m_{n-1}
\right\rVert^2\right]
\leq c_2 \lambda_{\max}(A)^2 (1 - \rho_1)^2 G^2
\left\lbrace
\sum_{k = 0}^{n-1} \rho_1^{n-1-k} 
\sqrt{\log T_{k+1}}
\right\rbrace^2
\leq c_2 (1 - \rho_1)(1 - \rho_1^n) G^2 \log T_n\eqsp.
\end{equation*}
\end{proof}

\subsection{Proof of Corollary \ref{cor:ams-grad}}

If $\gamma_{n}=c_{\gamma}n^{-\gamma}$ with $\gamma$ non negative, it comes that for all $n \geq 1$,
\[
\gamma_{n}^{2} - \gamma_{n+1}^{2} \leq c_{\gamma}^{2}2\gamma n^{-2 \gamma -1}\eqsp.
\]
Then,
\begin{align}
    \sum_{n = 1}^N  \left(1 - \frac{\gamma_{n+1}^2}{\gamma_n^2}\right) & = \sum_{n=1}^{N} \gamma_{n}^{-2} \left( \gamma_{n}^{2} - \gamma_{n+1}^{2} \right)   \leq 2\gamma \sum_{n=1}^{N} \frac{1}{n} = \mathcal{O} \left( \log (N) \right) \eqsp.
\end{align}


\section{Technical results}
\begin{definition}
A measurable function $f :\mathbb{R}^p\mapsto\mathbb{R}$ is said to satisfy the bounded difference condition if there exists non-negative constants $\kappa_1$, \ldots, $\kappa_p$ such that for all $\bx = (x_1, \ldots, x_p)$, $\by = (y_1, \ldots, y_p) \in \mathbb{R}^p$,
\begin{equation}
\label{eq:bd-cond}
    \left\lvert f(\bx) - f(\by) \right\rvert \leq \sum_{i = 1}^p \kappa_{i}\mathds{1}_{\{x_i\neq y_i\}}\eqsp.
\end{equation}
\end{definition}


\begin{lemma}
\label{lemma:mcdiarmid}
Assume that H\ref{hyp:stat}-H\ref{hyp:geomerg} hold. Then, for all $\theta\in\Theta$, for every geometrically recurrent small set $C$, there exists a positive constant $\beta$ such that for all $p\geq 1$, $t\in\mathbb{R}_{+}$, $x\in C$, and all measurable function $f:\mathbb{R}^p \mapsto \mathbb{R}$ satisfying the bounded difference condition \eqref{eq:bd-cond} with constants $(\kappa_{i})_{1\leq i \leq p}$, 
\begin{equation*}
    \mathbb{P}_{x}\left[ \left\lvert f(\bX_{1:p}) - \mathbb{E}_x\left[ f(\bX_{1:p}) \right] \right\rvert > t
    \right] \leq 2 \exp\left(-\frac{\beta t^2}{\sum_{i = 1}^p \kappa_i^2}\right)\eqsp.
\end{equation*}
\end{lemma}
\begin{proof}
The proof is given in Theorem 23.3.1 of \cite{douc2018markov}.
\end{proof}

\begin{lemma}
\label{lemma:diff-measure}
Assume that H\ref{hyp:stat}-H\ref{hyp:geomerg} hold. Then, for all $\theta\in\Theta$, for every geometrically recurrent small set $C$, there exist $\alpha \in [0,1)$ and $\zeta<\infty$ such that for all $p\geq 1$, $x\in C$, and all measurable function $f$ satisfying the bounded difference condition  \eqref{eq:bd-cond} with constants $(\kappa_{i})_{1\leq i \leq p}$, 
\begin{equation*}
    \left\lvert
    \mathbb{E}_{x}\left[f(\bX_{1:p})\right] - \mathbb{E}_{\pi_\theta}\left[f(\bX_{1:p})\right]
    \right\rvert \leq \zeta \sum_{i=0}^{p-1} \kappa_{i+1} \alpha^i\eqsp.
\end{equation*} 
\end{lemma}
\begin{proof}
The proof is given in Lemma 23.3.5 of \cite{douc2018markov}.
\end{proof}

\section{Additional experiments}
\label{sec:add_exp}

In this section, we conduct additional experiments on CIFAR-10 using Adagrad. We use the same model and experimental setup as in Section~\ref{sec:exp}, and use a decaying learning rate $\gamma_n = C_{\gamma}/\sqrt{n}$ with $C_{\gamma}=0.001$. Figures~\ref{fig:grad_norm_adagrad} and~\ref{fig:test_nll_adagrad} show the squared gradient norm and the test negative log-likelihood, respectively, as functions of epochs (left) and computational budget (right).
Overall, the conclusions are consistent with those obtained with AMSGrad in Section~\ref{sec:exp}. BR-IWAE and MLMC-IWAE show similar iteration-wise convergence in terms of the squared gradient norm. BR-IWAE achieves slightly lower test negative log-likelihood when evaluated per epoch, whereas MLMC-IWAE is more sample-efficient when comparing methods at matched computational budget.

\begin{figure}[ht]
  \begin{center}
    \centerline{\includegraphics[width=0.6\columnwidth]{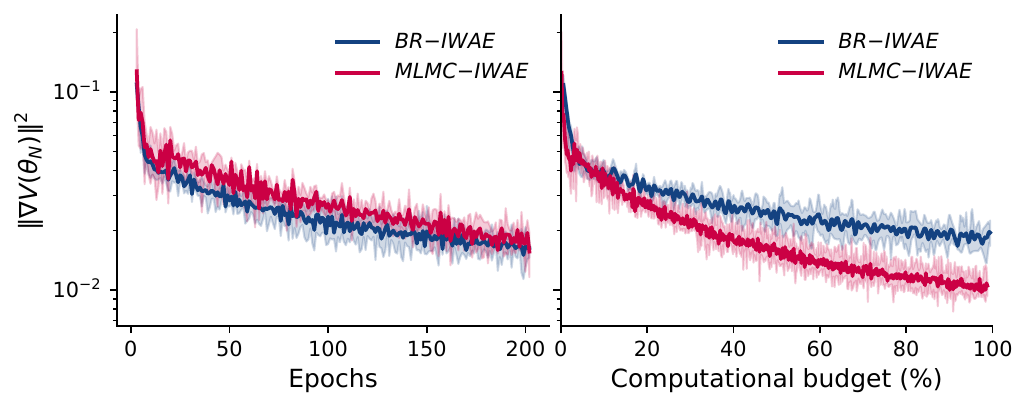}}
    \caption{
      Squared gradient norm $\|\nabla V(\theta_N)\|^2$ for BR-IWAE and MLMC-IWAE trained with Adagrad on CIFAR-10, shown as a function of epochs (left) and computational budget (right). Bold lines represent the mean over 5 independent runs.
    }
    \label{fig:grad_norm_adagrad}
  \end{center}
\end{figure}

\begin{figure}[ht]
  \begin{center}
    \centerline{\includegraphics[width=0.6\columnwidth]{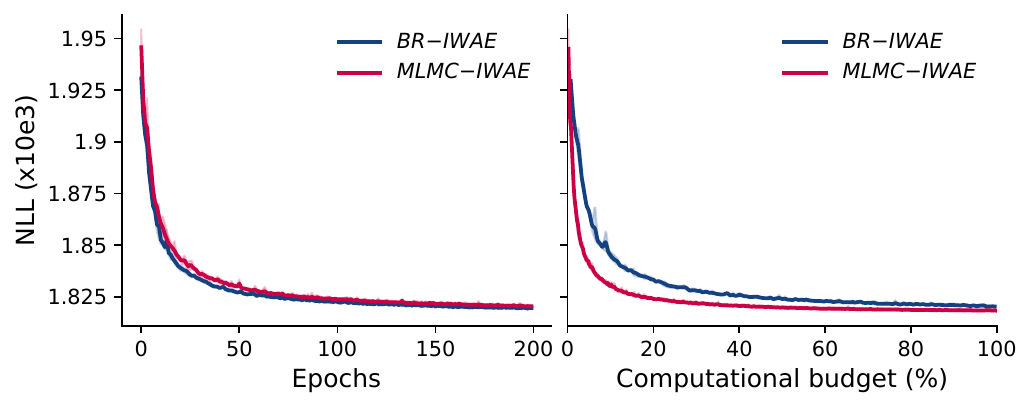}}
    \caption{
      Negative log-likelihood for BR-IWAE and MLMC-IWAE trained with Adagrad on CIFAR-10, shown as a function of epochs (left) and computational budget (right). Bold lines represent the mean over 5 independent runs.
    }
    \label{fig:test_nll_adagrad}
  \end{center}
\end{figure}

\end{document}